\documentclass[11pt]{amsart}

\usepackage{
    amsmath,
    amsfonts,
    amssymb,
    amsthm,
    amscd,
    comment,
    enumitem,
    etoolbox,
    gensymb,
    mathtools,
    mathdots,
    stmaryrd,
}
\usepackage[usenames,dvipsnames]{xcolor}
\usepackage{todonotes}
\usepackage[all]{xy}
\usepackage[hyphens]{url}
\usepackage[hyphenbreaks]{breakurl}

\usepackage[utf8]{inputenc}  
\usepackage{bbm}
\usepackage[colorlinks=true, linkcolor=blue, citecolor=blue, urlcolor=blue, breaklinks=true]{hyperref}

\DeclareFontFamily{OT1}{pzc}{}
\DeclareFontShape{OT1}{pzc}{m}{it}{<-> s * [1.10] pzcmi7t}{}
\DeclareMathAlphabet{\mathpzc}{OT1}{pzc}{m}{it}

\usepackage[capitalize]{cleveref} 
\crefname{defin}{Definition}{Definitions}
\crefname{eg}{Example}{Examples}
\crefname{egs}{Example}{Examples}
\crefname{lem}{Lemma}{Lemmas}
\crefname{theo}{Theorem}{Theorems}
\crefname{equation}{}{}
\crefname{enumi}{}{}

\newcommand{\C}{\mathbb{C}}
\newcommand{\N}{\mathbb{N}}

\newcommand{\D}{\mathbb{D}}

\newcommand{\kk}{\Bbbk}
\newcommand{\one}{\mathbbm{1}}
\newcommand{\CC}{\mathcal{C}}
\newcommand{\DD}{\mathcal{D}}
\newcommand{\B}{\mathcal{B}}
\newcommand{\AB}{\mathcal{AB}}
\newcommand{\OB}{\mathcal{OB}}
\newcommand{\AOB}{\mathcal{AOB}}
\renewcommand{\L}{\mathcal{L}} 

\newcommand{\g}{\mathfrak{g}}

\newcommand{\dashsmod}{\textup{-smod}}

\DeclareMathOperator{\op}{op}
\DeclareMathOperator{\End}{End}
\DeclareMathOperator{\Hom}{Hom}
\DeclareMathOperator{\ob}{ob}
\DeclareMathOperator{\tr}{tr}
\DeclareMathOperator{\Tr}{Tr}
\DeclareMathOperator{\str}{str}
\DeclareMathOperator{\Mat}{Mat}
\DeclareMathOperator{\id}{id}

\newcommand{\osp}{\mathfrak{osp}}
\newcommand{\SVec}{\mathpzc{SVec}} 

\newcommand{\scr}{\scriptstyle}

\newcommand{\bm}{\begin{bmatrix}}
\newcommand{\nm}{\end{bmatrix}}
\newcommand{\summ}{\sum\limits}

\newcommand{\go}{{\mathsf{I}}}
\newcommand{\goup}{\uparrow}
\newcommand{\godown}{\downarrow}

\usepackage{tikz}
\usetikzlibrary{arrows.meta}
\usetikzlibrary{decorations.markings}
\usetikzlibrary{calc}

\tikzset{anchorbase/.style={>=To,baseline={([yshift=-0.5ex]current bounding box.center)}}}
\tikzset{ 
    centerzero/.style={>=To,baseline={([yshift=-0.5ex](#1))}},
    centerzero/.default={0,0}
}
\tikzset{wipe/.style={white,line width=3pt}}

\newcommand\dotlabel[1]{$\scriptstyle{#1}$}

\newcommand\adot[1]{\filldraw[fill=white, draw=black] #1 circle (1.5pt)}
\newcommand\bdot[1]{\filldraw[fill=black, draw=black] #1 circle (1.5pt)}
\newcommand{\ibubclock}[1]{\draw[->,black,fill=Thistle] #1++(.15,0) arc(0:-360:0.15);}
\newcommand{\ibubcounter}[1]{\draw[->,black,fill=YellowOrange] #1++(.15,0) arc(0:360:0.15);}
\newcommand{\purpledot}[1]{\ibubclock{#1}}
\newcommand{\orangedot}[1]{\ibubcounter{#1}}
\newcommand\token[3]{
    \filldraw[blue] (#1) circle (1.5pt) node[anchor=#2] {\dotlabel{#3}}
}
\newcommand\link[2]{
	\draw (#1) -- (#2);
	\draw[fill=white, draw=black] (#1) circle (1.5pt);
	\draw[fill=white, draw=black] (#2) circle (1.5pt);
}
\newcommand\dottedlink[2]{
	\draw[densely dotted] (#1) -- (#2);
	\draw[fill=white, draw=black] (#1) circle (1.5pt);
	\draw[fill=white, draw=black] (#2) circle (1.5pt);
}
\newcommand\triangleUp[1]{\filldraw[blue] ($#1 + (0,0.075) $) -- ++(0.075,-0.13) -- ++(-0.15,0) -- cycle}
\newcommand\triangleDown[1]{\filldraw[blue] ($#1 + (0,-0.075) $) -- ++(0.075,0.13) -- ++(-0.15,0) -- cycle}
\newcommand\triangleRight[1]{\filldraw[blue] ($#1 + (0.075,0) $) -- ++(-0.13,0.075) -- ++(0,-0.15) -- cycle}

\newcommand\teleportUU[2]{
	\draw[blue] (#1) -- (#2);
	\triangleUp{(#1)};
	\triangleUp{(#2)};
}
\newcommand\teleportDD[2]{
	\draw[blue] (#1) -- (#2);
	\triangleDown{(#1)};
	\triangleDown{(#2)};
}
\newcommand\teleportUD[2]{
	\draw[blue] (#1) -- (#2);
	\triangleUp{(#1)};
	\triangleDown{(#2)};
}
\newcommand\teleportDU[2]{
	\draw[blue] (#1) -- (#2);
	\triangleDown{(#1)};
	\triangleUp{(#2)};
}
\newcommand\teleportRR[2]{
	\draw[blue] (#1) -- (#2);
	\triangleRight{(#1)};
	\triangleRight{(#2)};
}

\newtheorem{theo}{Theorem}[section]
\newtheorem{prop}[theo]{Proposition}
\newtheorem{lem}[theo]{Lemma}

\newtheorem{conj}[theo]{Conjecture}

\theoremstyle{definition}
\newtheorem{defin}[theo]{Definition}
\newtheorem{rem}[theo]{Remark}
\newtheorem{eg}[theo]{Example}
\newtheorem{egs}[theo]{Examples}

\numberwithin{equation}{section}
\allowdisplaybreaks

\setenumerate[1]{label=(\alph*)}

\begin{document}

\title{Affine Frobenius Brauer Categories}
\author{Saima Samchuck-Schnarch}
\address{
	Department of Mathematics and Statistics \\
	University of Ottawa \\
	Ottawa, ON K1N 6N5, Canada
}
\email{ssamc090@uottawa.ca}



\maketitle
\thispagestyle{empty}

\begin{abstract}
	We define the \emph{affine Frobenius Brauer categories} associated to each symmetric involutive Frobenius superalgebra $A$. We then define an action of these categories on the categories of finite-dimensional supermodules for orthosymplectic Lie superalgebras defined over $A$. When $A$ is the base field, we recover the previously-studied affine Brauer category; for other choices of $A$, the categories are novel. Finally, we state a conjecture for bases of homomorphism spaces in affine Frobenius Brauer categories, and outline a potential proof strategy.
\end{abstract}

\setcounter{tocdepth}{2}
\tableofcontents

\section{Introduction}

Brauer algebras were introduced by the eponymous Richard Brauer in \cite{brauerAlgebras}. These algebras are the analogues, for orthogonal and symplectic (and, more generally, orthosymplectic) groups, of the group algebras of symmetric groups that feature in Schur--Weyl duality. The \emph{Brauer category} $\B$, studied extensively by Lehrer and Zhang in \cite{brauerCategory}, is the free linear symmetric monoidal category generated by a single symmetrically self-dual object $\go$; its name reflects the fact that the endomorphism algebras of $\B$ are Brauer algebras. Many results involving Brauer algebras have natural interpretations in terms of $\B$; for instance, Schur--Weyl duality for orthosymplectic groups can be proved by constructing full functors from $\B$ to the categories of finite-dimensional representations of these groups. Many variants of the Brauer category have been defined and studied. The \emph{oriented Brauer category} $\OB$ is the free symmetric monoidal category generated by a single object $\uparrow$ and its dual $\downarrow$, and it is an analogue of $\B$ corresponding to the general linear groups and Lie algebras. More generally, \emph{oriented Frobenius Brauer categories}, denoted $\OB(A)$ (first implicitly appearing as a subcategory of the Frobenius Heisenberg category $\mathcal{H}eis_{A, 0}$ in \cite{Savage_2019}, and then explicitly defined in \cite{mcsween2021affine}), correspond to general linear groups and Lie superalgebras defined over a Frobenius superalgebra $A$. In the unoriented case, one needs to restrict to Frobenius superalgebras equipped with an involution $-^\star$, leading to the Frobenius Brauer categories $\B^\sigma(A, -^\star)$ of \cite{diagrammaticsRealSupergroups}. Oriented and unoriented Frobenius Brauer categories provide a natural framework for proving results about the representation theory of the classical Lie algebras and their Frobenius superalgebra analogues.

There are also the \emph{affine Brauer} and \emph{affine oriented Brauer categories} $\AB$ and $\AOB$, introduced in \cite{ruiSong2018affine} and \cite{Brundan_2017}, respectively. These affine categories act via translation functors on the categories of finite-dimensional supermodules over orthosymplectic and general linear Lie superalgebras, respectively. The affine oriented Brauer category has been generalized to the case of Frobenius algebras. Affine oriented Frobenius Brauer categories $\AOB(A)$ first appeared as a special case of the aforementioned Frobenius Heisenberg categories in \cite{Savage_2019} (namely, the case of central charge $k = 0$), and then were studied more explicitly in \cite{mcsween2021affine}. In the current paper, we will define the \emph{(unoriented) affine Frobenius Brauer category} $\AB(A, -^\star)$ and prove in Theorem \ref{thm:affineFunctor} that it acts on the category of finite-dimensional supermodules for orthosymplectic Lie superalgebras defined over $A$.

The introduction of Frobenius superalgebras to Brauer categories is part of a broader wave of research; many algebraic objects can be naturally generalized by introducing an appropriate sort of algebra into the definition. For instance, Frobenius generalizations of degenerate and non-degenerate affine Hecke algebras were respectively studied in \cite{affineWreathProductAlgebras} and \cite{quantumAffineWreathAlgebras}, then extended to the case of nilHecke algebras in \cite{frobeniusNilHecke}. Along similar lines, Kleshchev and Muth studied generalized Schur superalgebras defined over unital superalgebras in \cite{generalizedSchurAlgebras} and \cite{schurifying}. These sorts of generalizations have also appeared in quantum topology, e.g.\ Khovanov's study of Frobenius algebra enrichments of link homologies in \cite{linkHomology}. Such generalizations often serve to unify existing definitions and results, and help to make clear how proofs and ideas from one case can be applied more generally to obtain novel results. In the following paragraph, we outline one example of this approach in action.

One important application of Frobenius Brauer categories is to the representation theory of real forms of general linear, orthosymplectic, periplectic, and isomeric Lie superalgebras. Such Lie superalgebras have been studied extensively over $\C$, but relatively little is known about the real case. In \cite{diagrammaticsRealSupergroups}, the \emph{incarnation superfunctor} we mention in Proposition \ref{prop:incarnation} and Theorem \ref{thm:affineFunctor} of the current paper was proved to be full when $A$ is a central real division superalgebra, and similarly for the corresponding oriented incarnation superfunctor. Equivalences between various categories of tensor supermodules for real supergroups, analogous to results previously known over $\C$, follow from this fullness; see \cite[Prop.~11.5, 12.5, 13.5]{diagrammaticsRealSupergroups}. The affine Frobenius Brauer categories introduced in the present paper provide further tools for studying real supergroups and real forms of Lie superalgebras.

In the final section of this paper, we will state a conjecture for bases of the homomorphism spaces in $\AB(A, -^\star)$ and sketch a potential method of proof. This basis conjecture generalizes the known result \cite[Thm.~B]{ruiSong2018affine} for $\AB$, and the proof technique draws inspiration from methods used to prove basis results for $\AOB(A)$, $\B(A, -^\star)$, and the nil-Brauer category $\mathcal{NB}_t$ in \cite{Brundan_2021}, \cite{diagrammaticsRealSupergroups}, and \cite{nilBrauer}, respectively.

We conclude this introduction by listing some potential future avenues of research related to affine Frobenius Brauer categories.
\begin{itemize}
	\item One can adjust the definitions of the oriented and unoriented Frobenius Brauer categories to allow for the case of non-symmetric Frobenius superalgebras such as the two-dimensional Clifford superalgebra; this level of generality was addressed in the oriented case in \cite{Savage_2019}, and in the unoriented case in \cite{diagrammaticsRealSupergroups}. The author of the present paper expects that it should be possible to find an appropriate generalization of the unoriented affine Frobenius Brauer category $\AB(A, -^\star)$ such that the corresponding version of Theorem \ref{thm:affineFunctor} holds when $A$ is non-symmetric.
	\item As we will discuss in Remark \ref{rem:oddCase}, the most straightforward generalization of the affine functor $\hat{F}_\Phi$ from Theorem \ref{thm:affineFunctor} does not yield a useful action of the affine dot when considering odd supersymmetric forms. Similar issues previously appeared in the study of the representation theory of periplectic Lie superalgebras $\mathfrak{p}(n)$. In \cite[\S4]{balagovic2019}, Balagovic et al. defined the \emph{fake Casimir elements} $\Omega \in \mathfrak{p}(n) \otimes \mathfrak{p}(n)^\perp$. These fake Casimir elements are used in \cite{affineVW} to define an action of the affine dot that appears in the \emph{affine VW category}. The affine VW category is an analogue of the affine Brauer category $\AB$, defined with respect to an odd form. The defining relations for the affine VW category are similar to those for the affine Frobenius Brauer categories studied in the present paper, but the relations for sliding dots over caps, respectively \eqref{rel:dotCapSlide} and \cite[(R4)]{affineVW}, are quite different, leading to categories with distinct structures. It seems plausible that one could define a Frobenius superalgebra version of the affine VW category, using generalized fake Casimir elements to produce an appropriate action of the affine dot. This would enable further study of real forms of periplectic Lie superalgebras, extending the diagrammatic tools introduced in \cite{diagrammaticsRealSupergroups}.
	\item The \emph{BMW category}, also known as the \emph{quantum Brauer category} or \emph{Kauffman skein category}, is a quantum generalization of the unoriented Brauer category. Its endomorphism algebras are the BMW algebras introduced independently in \cite{BWalgebras} and \cite{Malgebras}, and it has been used to study the representation theory of the quantum enveloping algebras of the special orthogonal and symplectic Lie algebras. Affine BMW algebras were defined in \cite{orellana2004affine}, and a corresponding affine version of the Kauffman category first appeared in the literature in \cite{affineKauffman}. Oriented quantum affine Frobenius Brauer categories are a special case of the quantum Frobenius Heisenberg categories defined in \cite{Brundan_2022}. In the quantum setting, one can define affine categories in terms of (non-symmetric) braided string diagrams wrapped around a cylinder; see e.g.\ \cite{affinizationMonoidalCategories} for details. This topological approach provides an alternate perspective that is not available in the case of symmetric monoidal categories. Defining and studying \emph{unoriented} quantum affine Frobenius categories would be a natural generalization of the work in the present paper. The affine Kauffman/BMW category would be one example of such a category, where the Frobenius algebra is the base field. In general, unoriented quantum affine Frobenius categories would naturally correspond to as-yet-undefined quantum enveloping algebras of orthosymplectic Lie algebras defined over Frobenius algebras. As such, the definition and study of such categories could lead to the discovery of new enveloping algebras. 
\end{itemize}

\section*{Acknowledgements}

This research was supported by a PGS-D scholarship from the Natural Sciences and Engineering Research Council of Canada (NSERC). Thanks to Alistair Savage for his support, guidance, and productive discussions; to Alexandra McSween for her help in understanding oriented Frobenius Brauer categories; to Léo Schelstraete for pointing out an issue with an earlier version of Conjecture \ref{con:basisConjecture}; and to the referee for their detailed and helpful comments and suggestions.

\section{Preliminaries}

\subsection{Conventions}

Throughout this paper, $\kk$ is a field of characteristic different from 2. All vector spaces and tensor products are taken over $\kk$.

If $V$ is a super vector space and $v \in V$ is a homogeneous element, we write $\bar{v}$ for the parity of that element, i.e.\ $\bar{v} = 0$ if $v$ is even, and $\bar{v} = 1$ if $v$ is odd. Whenever we write an expression involving terms of the form $\bar{v}$, we are implicitly assuming/requiring that $v$ is homogeneous. When definitions or proofs are given in terms of homogeneous elements, the full definition or proof follows from linear extension to the whole super vector space.

We write $\one$ for the unit object of a monoidal supercategory, and $\id_X$ for the identity morphism of an object $X$.

Except for Lie superalgebras, all superalgebras in this paper are assumed to be associative and unital.

\subsection{Supercategories}

Much of this paper is concerned with (strict) monoidal supercategories and their associated calculus of string diagrams. We will recall a few key properties here; for full details, see e.g.\ \cite{monoidalSupercategories} and \cite[\S 2, \S 3.1]{saimaMSCThesis}.

We write $\SVec$ for the category whose objects are super vector spaces and whose morphisms are parity-preserving linear maps. A \emph{supercategory} is a category enriched in $\SVec$. Thus a supercategory's hom-sets are super vector spaces, and its composition is parity-preserving, i.e.\ $\overline{f \circ g} = \bar{f} + \bar{g}$. A \emph{superfunctor} is a $\kk$-linear functor between supercategories that preserves the parity of morphisms. A \emph{supernatural transformation} $\alpha \colon F \to G$ of parity $i \in \{0, 1\}$ between two superfunctors $F, G \colon \CC \to \DD$ is a collection of $\DD$-morphisms $\alpha_X \colon F(X) \to G(X)$, ranging over $X \in \ob(\CC)$, such that $\overline{\alpha_X} = i$ for all $X$ and satisfying $G(f) \circ \alpha_X = \alpha_Y \circ (-1)^{i\bar{f}}F(f)$ for all $\CC$-morphisms $f \colon X \to Y$. Note that even supernatural transformations are ordinary natural transformations (all of whose component maps are even), but odd supernatural transformations are only natural transformations up to a sign. A general \emph{supernatural transformation} $\alpha \colon F \to G$ is a sum of an even and an odd supernatural transformation.

Given a supercategory $\CC$ (not necessarily monoidal), the \emph{endofunctor supercategory} $\End(\CC)$ is a strict monoidal supercategory, with the composition and tensor product in $\End(\CC)$ respectively being given by vertical and horizontal composition of supernatural transformations.

We represent morphisms in strict monoidal supercategories via string diagrams, with composition corresponding to vertical stacking and tensor products corresponding to horizontal juxtaposition. The main feature distinguishing string diagrams in the super setting from those in the non-super setting is the existence of the \emph{superinterachange law} for monoidal supercategories: for all morphisms $f \colon X \to Y$ and $g \colon A \to B$, we have 
\begin{equation}
(f \otimes \id_B) \circ (\id_X \otimes g)  = f \otimes g = (-1)^{\overline{f} \overline{g}} (\id_Y \otimes g) \circ (f \otimes \id_A),
\end{equation}
which in the strict case can be drawn as:
\begin{equation}\label{eq:superInterchangeLaw}  \begin{tikzpicture}[anchorbase]
\draw[-] (0, -0.25) -- (0, 1.25);
\draw[-] (0.5, -0.25) -- (0.5, 1.25);
\filldraw[fill=white,draw=black] (0,0.75) circle (5pt);
\node at (0,0.75) {$\scriptstyle{f}$};
\filldraw[fill=white,draw=black] (0.5,0.25) circle (5pt);
\node at (0.5,0.25) {$\scriptstyle{g}$};
\node at (0,-0.5) {$X$};
\node at (0,1.5) {$Y$};
\node at (0.5,-0.5) {$A$};
\node at (0.5,1.5) {$B$};
\end{tikzpicture} \quad = \quad \begin{tikzpicture}[anchorbase]
\draw[-] (0, -0.25) -- (0, 1.25);
\draw[-] (0.5, -0.25) -- (0.5, 1.25);
\filldraw[fill=white,draw=black] (0,0.5) circle (5pt);
\node at (0,0.5) {$\scriptstyle{f}$};
\filldraw[fill=white,draw=black] (0.5,0.5) circle (5pt);
\node at (0.5,0.5) {$\scriptstyle{g}$};
\node at (0,-0.5) {$X$};
\node at (0,1.5) {$Y$};
\node at (0.5,-0.5) {$A$};
\node at (0.5,1.5) {$B$};
\end{tikzpicture} \quad = \quad (-1)^{\overline{f} \overline{g}}
\begin{tikzpicture}[anchorbase]
\draw[-] (0, -0.25) -- (0, 1.25);
\draw[-] (0.5, -0.25) -- (0.5, 1.25);
\filldraw[fill=white,draw=black] (0,0.25) circle (5pt);
\node at (0,0.25) {$\scriptstyle{f}$};
\filldraw[fill=white,draw=black] (0.5,0.75) circle (5pt);
\node at (0.5,0.75) {$\scriptstyle{g}$};
\node at (0,-0.5) {$X$};
\node at (0,1.5) {$Y$};
\node at (0.5,-0.5) {$A$};
\node at (0.5,1.5) {$B$};
\end{tikzpicture}\ .\end{equation}
Note that if either $f$ or $g$ is even, the sign vanishes and we recover the ordinary interchange law for monoidal categories.

\subsection{Frobenius Superalgebras} \label{ss:frobeniusSuperalgebras}

\begin{defin} \label{def:frobeniusSuperalgebra}
	A \emph{Frobenius superalgebra} is a finite-dimensional superalgebra $A$ equipped with a linear functional $\tr \colon A \to \kk$, called the \emph{trace map} for $A$, such that the induced bilinear form $(x, y) \mapsto \tr(xy)$ is nondegenerate. (Equivalently, this says that $\ker(\tr)$ contains no nonzero left ideals.) We call $A$ a \emph{symmetric} Frobenius superalgebra if its trace map satisfies \begin{equation}\label{eq:symmetric}
	\tr(xy) = (-1)^{\bar{x}\bar{y}}\tr(yx)
	\end{equation}
	for all $x, y \in A$.
\end{defin}

Throughout this paper, we require that all trace maps are even.

Given a basis $B_A$ for a symmetric Frobenius superalgebra $A$, there is an associated \emph{left dual basis}, denoted $B_A^\vee = \{b^\vee \mid b \in B_A \}$, satisfying $\tr(b^\vee c) = \delta_{bc}$ for all $b, c \in B_A$. Note that $\overline{b^\vee} = \bar{b}$. It is straightforward to show that
\begin{equation} \label{eq:doubleDual}
(b^\vee)^\vee  = (-1)^{\bar{b}}b
\end{equation} 
for all $b \in B_A$, where on the left hand side we are taking duals with respect to $B_A^\vee$. We also have
\begin{equation} \label{eq:traceDecomposition}
	a = \summ_{b \in B_A} \tr(b^\vee a)b = \summ_{b \in B_A} \tr(ab)b^\vee \text{ for all } a \in A.
\end{equation}

\begin{defin}
	An \emph{involution} on a superalgebra $A$ is an even self-inverse $\kk$-linear endomorphism $-^\star \colon A \to A$ satisfying $(xy)^\star = (-1)^{\bar{x}\bar{y}}y^\star x^\star$ for all $x, y \in A$. Equivalently, an involution is an even self-inverse algebra homomorphism from $A$ to $A^{\op}$. (Note that some authors refer to such maps as \emph{anti-involutions}, and instead use the term ``involution'' to refer to self-inverse maps satisfying $(xy)^\star = x^\star y^\star$.)
	
	An involution $-^\star$ on a Frobenius superalgebra $(A, \tr)$ is said to be \emph{compatible with the trace map on $A$} if it satisfies 
	\begin{equation} \label{eq:traceCompatibility}
	\tr(x^\star) = \tr(x)
	\end{equation} for all $x \in A$.
	
	A \emph{symmetric involutive Frobenius superalgebra} is a symmetric Frobenius superalgebra equipped with a compatible involution $-^\star$.
\end{defin}

\begin{egs} \label{ex:basicFrobenius}
	The (purely even) two-dimensional real algebra $\C$ becomes a symmetric involutive Frobenius superalgebra when equipped with the trace map $\tr(x + iy) = x$ and the involution given by complex conjugation. Similarly, the (purely even) real quaternion algebra $$\mathbb{H} = \langle i, j, k \mid i^2 = j^2 = k^2 = ijk = -1 \rangle$$ is a symmetric involutive Frobenius superalgebra with respect to the trace map $\tr(a + ib + jc + kd) = a$ and the involution $(a + ib + jc + kd)^\star = a - ib - jc - kd$.
	
	Further examples of symmetric involutive Frobenius superalgebras include finite group algebras, zigzag superalgebras, and truncated polynomial algebras.
\end{egs}

\begin{eg} \label{ex:frobeniusMatrix}
	If $(A, \tr_A, -^\star)$ is any involutive Frobenius superalgebra, the matrix superalgebra $\Mat_{m \mid n}(A)$ is itself a Frobenius superalgebra with trace map $\tr = \tr_A \circ \str$. Here, $\str$ denotes the \emph{supertrace}, given on a supermatrix in block form by $$\str\left(\bm M_{00} & M_{01} \\ M_{10} & M_{11} \nm\right) = \Tr(M_{00}) - (-1)^{\overline{M_{11}}}\Tr(M_{11}),$$ where $\Tr$ is the ordinary matrix trace. This Frobenius superalgebra is symmetric if and only if $A$ is. When $n$ is even, there is always a compatible involution on $\Mat_{m \mid n}(A)$ called the \emph{generalized orthosymplectic involution}. Such an involution can also be constructed when $n$ is odd for certain choices of $A$; see Remark \ref{rem:generalizedInvolution} for further details.
\end{eg}

\begin{lem}
	Let $(A, \tr, -^\star)$ be a symmetric involutive Frobenius superalgebra. Let $B_A$ be a basis of $A$. For all $x, y \in A$ and $b \in B_A$, we have:
	\begin{equation} \label{eq:traceDoubleStar}
	\tr(xy) = \tr(x^\star y^\star), \quad \quad \tr(x^\star y) = \tr(x y^\star),
	\end{equation}
	\begin{equation} \label{eq:dualStar}
	(b^\star)^\vee = (b^\vee)^\star,
	\end{equation}
	where on the left of \eqref{eq:dualStar} we are taking duals with respect to the basis $B_A^\star = \{b^\star \mid b \in B_A \}$.
\end{lem}

\begin{proof}
	For the first identity, we have
	$$ \tr(xy) \overset{\eqref{eq:traceCompatibility}}{=} \tr((xy)^\star)
	= (-1)^{\bar{x}\bar{y}}\tr(y^\star x^\star)
	\overset{\eqref{eq:symmetric}}{=} \tr(x^\star y^\star).$$
	The second identity follows from the fact that $-^\star$ is self-inverse. For the third identity, we have the following for all $b, c \in B_A$:
	$$ \tr((c^\vee)^\star b^\star)
	\overset{\eqref{eq:traceDoubleStar}}{=} \tr(c^\vee b) = \delta_{bc},$$
	as desired.
\end{proof}

\subsection{Supermodules and Lie Superalgebras} \label{ss:supermodules}

For the rest of the paper, fix a symmetric involutive Frobenius superalgebra $(A, \tr, -^\star)$ and a homogeneous $\kk$-basis $B_A$ of $A$. For the rest of this section, let $V$ be a right $A$-supermodule.

\begin{defin} \label{def:naturalSupermodule}
	Let $m, n \in \N$. We write $A^{m \mid n}$ for the right $A$-supermodule that is equal to $A^{m + n}$ as a module, with element parities determined by $\overline{ae_i} = \bar{a} + p(i)$ for $a \in A$ and $1 \leq i \leq m + n$, where $e_i$ denotes the vector with a 1 in position $i$ and zeroes elsewhere, and $$p(i) = \begin{cases} 0 &\text{if } 1 \leq i \leq m, \\ 1 &\text{if } m + 1 \leq i \leq m + n. \end{cases}$$
\end{defin}

\begin{defin}
	A \emph{$-^\star$-sesquilinear form} on $V$ is an even $\kk$-bilinear map $\varphi \colon V \times V \to A$ that satisfies the following identity for all $a, b \in A$ and $v, w \in V$:
	\begin{equation}
	\varphi(va, wb) = (-1)^{\bar{a}\bar{v}}a^\star \varphi(v, w)b.
	\end{equation}
	Let $\nu \in \{1, -1\}$. A \emph{($\nu, -^\star$)-superhermitian form} on $V$ is a $-^\star$-sesquilinear form $\varphi$ that additionally satisfies the following identity for all $v, w \in V$:
	\begin{equation} \label{def:superhermitian}
	\varphi(v, w) = \nu(-1)^{\bar{v}\bar{w}}\varphi(w, v)^\star.
	\end{equation}
	A $-^\star$-sesquilinear form $\varphi$ is called \emph{nondegenerate} if the map $v \mapsto \varphi(v, -)$ is an injective $A$-supermodule homomorphism $V \to \Hom_A(V, A)$, and \emph{unimodular} if that map is an $A$-supermodule isomorphism.
\end{defin}

Note that we require sesquilinear forms (and supersymmetric forms, which we will define shortly) to be even in this paper; see Remark \ref{rem:oddCase} for a discussion of why we exclude odd forms.

\begin{eg} \label{ex:hermitian}
	Let $-^\star \colon \C \to \C$ denote complex conjugation, and suppose $V$ is a purely even right $\C$-supermodule. Then a $(1, -^\star)$-superhermitian form on $V$ is an ordinary hermitian form, and a $(-1, -^\star)$-superhermitian form is a skew-hermitian form. If we instead use $\id$ as our involution, a $(1, \id)$-superhermitian form is a symmetric $\C$-bilinear form, and a $(-1, \id)$-superhermitian form is a skew-symmetric $\C$-bilinear form.
\end{eg}

\begin{lem}[{\cite[Lem.~7.11, Lem.~7.12]{diagrammaticsRealSupergroups}}] \label{lem:generalizedOrthosymplecticInvolution}
	Suppose $\varphi$ is a unimodular superhermitian form on $V$. For all $X \in \End_A(V)$, there exists a unique $X^\dagger \in \End_A(V)$, called the map \emph{adjoint} to $X$, that satisfies the following identity for all $v, w \in V$:
	\begin{equation} \label{def:dagger}
	\varphi(v, Xw) = (-1)^{\bar{X}\bar{v}}\varphi(X^\dagger v, w).
	\end{equation}
	Moreover, the map $X \to X^\dagger$ is an involution of the superalgebra $\End_A(V)$ known as the \emph{generalized orthosymplectic involution}.
\end{lem}

\begin{rem} \label{rem:generalizedInvolution}
	For any symmetric involutive Frobenius superalgebra $(A, \tr, -^\star)$, one can define a $(1, -^\star)$-superhermitian form on $A^{m \mid n}$ when $n$ is even; see \cite[\S5.4]{saimaMSCThesis} for details. For some choices of $A$, one can define superhermitian forms on $A^{m \mid n}$ when $n$ is odd as well; see \cite[\S{A.3}, \S{A.4}, \S{A.5}]{diagrammaticsRealSupergroups} for some examples. In either case, after identifying $\End_A(A^{m \mid n})$ and $\Mat_{m \mid n}(A)$ in the usual way, Lemma \ref{lem:generalizedOrthosymplecticInvolution} yields an involution on $\Mat_{m \mid n}(A)$. We will prove that generalized orthosymplectic involutions are compatible with supertraces in Lemma \ref{lem:strNondegenerate1}.
\end{rem}

\begin{defin} \label{def:orthosymplecticLieSuperalgebra}
	Let $\varphi$ be a unimodular superhermitian form on $V$. The \emph{orthosymplectic Lie superalgebra associated to $\varphi$ and $A$} is defined as
	\begin{equation} \label{def:osp}
	\osp(\varphi) = \{X \in \End_A(V) \mid X^\dagger = -X \},
	\end{equation}
	with Lie superbracket given by the supercommutator, i.e.\ $[X, Y] = XY - (-1)^{\bar{X}\bar{Y}}YX$.
\end{defin}

It is straightforward to show that
$$\osp(\varphi) = \{X \in \End_A(V) \mid \varphi(Xv, w) = -(-1)^{\bar{X}\bar{v}}\varphi(v, Xw)\text{ for all } v, w \in V \}.$$ Note that when $A = \kk$, Definition \ref{def:orthosymplecticLieSuperalgebra} recovers the usual definition for the orthosymplectic Lie superalgebra associated to a nondegenerate supersymmetric form $\varphi$.

\begin{defin}
	Let $\nu \in \{1, -1\}$. A \emph{$(\nu, -^\star)$-supersymmetric form} on $V$ is an even $\kk$-bilinear map $\Phi \colon V \times V \to \kk$ that satisfies the following identities for all $a \in A$ and $v, w \in V$:
	\begin{equation} \label{def:superSymmetric}
	\Phi(v, w) = \nu(-1)^{\bar{v}\bar{w}}\Phi(w, v), \quad \quad \Phi(va, w) = (-1)^{\bar{a}\bar{w}}\Phi(v, wa^\star).
	\end{equation}
\end{defin}

\begin{lem}[{\cite[Lem.~7.10]{diagrammaticsRealSupergroups}}] \label{lem:hermitianToSymmetric}
	If $\varphi$ is a nondegenerate $(\nu, -^\star)$-superhermitian form on $V$, then the map $\Phi := \tr\circ\,\varphi$ is a nondegenerate $(\nu, -^\star)$-supersymmetric form on $V$.
\end{lem}

For the remainder of this section, we specialize to the case $V = A^{m \mid n}$ for some $m, n \in \N$, and fix a homogeneous $A$-basis $B_V$ for $V$.

\begin{defin} \label{def:supertrace}
	Let $\varphi$ be a unimodular $(\nu, -^\star)$-superhermitian form on $V$. The \emph{supertrace with respect to $\varphi$} is the map $\str_\varphi \colon \End_A(V) \to A$ given by:
	$$\str_\varphi(X) = \summ_{b \in B_V}\varphi(Xb, b^\vee).$$
	(The duals for $B_V$ are taken with respect to $\varphi$, satisfying $\varphi(b^\vee, c) = \delta_{bc}$ for all $b, c \in B_V$.)
\end{defin}

Quick calculations show that $\str_\varphi$ is independent of the choice of basis $B_V$, and that for all $b \in B_V$,
\begin{equation} \label{eq:doubleDualV}
	(b^\vee)^\vee = \nu (-1)^{\bar{b}} b.
\end{equation}

\begin{lem} \label{lem:strNondegenerate1}
	If $\varphi$ is a unimodular $(\nu, -^\star)$-superhermitian form on $V$, then $(X, Y) \mapsto \tr(\str_\varphi(XY))$ is a $(\nu, -^\star)$-supersymmetric nondegenerate bilinear form on $\End_A(V)$, and for all $X \in \End_A(V)$,
	\begin{equation} \label{eq:strDagger}
	\str_\varphi(X^\dagger) = \str_\varphi(X).
	\end{equation}
\end{lem}

\begin{proof}
	The nondegeneracy and supersymmetry of $(X, Y) \mapsto \tr(\str_\varphi(XY))$ follow from the fact that $\varphi$ is unimodular and superhermitian. For $X \in \End_A(V)$, we compute:
	\begin{multline*}
	\tr(\str_\varphi(X^\dagger))
	= 
	\tr\left(\summ_{b \in B_V} \varphi(X^\dagger b, b^\vee)\right)
	\overset{\eqref{def:dagger}}{=}
	\tr\left(\summ_{b \in B_V} (-1)^{\bar{X}\bar{b}}\varphi(b, X b^\vee)\right)
	\\
	\overset{\eqref{def:superhermitian}}{=}
	\tr\left(\summ_{b \in B_V} \nu(-1)^{\bar{b}}\varphi(X b^\vee, b)^\star\right)
	\overset{\eqref{eq:doubleDualV}}{=}
	\tr\left(\summ_{b \in B_V} \varphi(X b, b^\vee)^\star \right)
	\\
	\overset{\eqref{eq:traceCompatibility}}{=}
	\tr\left(\summ_{b \in B_V} \varphi(X b, b^\vee)\right)
	= \tr(\str_\varphi(X)),
	\end{multline*}
	switching to a sum over the dual basis in the fourth equality.
\end{proof}

\begin{lem} \label{lem:strNondegenerate2}
	If $\varphi$ is a unimodular $(\nu, -^\star)$-superhermitian form on $V$, then $(X, Y) \mapsto \tr(\str_\varphi(XY))$ is a nondegenerate bilinear form on $\osp(\varphi)$.
\end{lem}

\begin{proof}
	Let $X \in \osp(\varphi)$ be a nonzero element. By the definition of $\osp(\varphi)$, we have $X^\dagger = -X$, and hence $X - X^\dagger = 2X$ is nonzero. By Lemma \ref{lem:strNondegenerate1}, there exists some $y \in \End_A(V)$ such that $\tr(\str_\varphi((X - X^\dagger)y) \neq 0$. Set $Y = y - y^\dagger$. Since $-^\dagger$ is an involution, we have that $Y^\dagger = -Y$, i.e.\ $Y \in \osp(\varphi)$. Then:
	\begin{multline*}
		\tr(\str_\varphi(XY))
		=
		\tr(\str_\varphi(Xy)) - \tr(\str_\varphi(Xy^\dagger))
		\overset{\eqref{eq:strDagger}}{=}
		\tr(\str_\varphi(Xy)) - \tr(\str_\varphi((Xy^\dagger)^\dagger))
		\\
		=
		\tr(\str_\varphi(Xy)) - \tr(\str_\varphi(X^\dagger y))
		=
		\tr(\str_\varphi((X - X^\dagger)y)
		\neq
		0,
	\end{multline*}	
	showing the desired nondegeneracy.
\end{proof}

\section{Frobenius Brauer Categories} \label{s:brauerCategories}

In this section, we recall the definition and key properties of the Frobenius Brauer categories $\B(A, -^\star)$, and then define and study \emph{teleporter} morphisms in those categories. We use these teleporters to define the affine Frobenius Brauer categories $\AB(A, -^\star)$, and then prove Theorem~\ref{thm:affineFunctor}, which states that these affine categories have a natural action on the corresponding categories of supermodules for orthosymplectic Lie superalgebras defined over $A$.

\subsection{Basic Properties and the Incarnation Superfunctor}

\begin{defin}[{\cite[Def.~9.1]{diagrammaticsRealSupergroups}}] \label{def:frobeniusBrauerCategory}
	The \emph{Frobenius Brauer category associated to $(A, -^\star)$}, denoted $\B(A, -^\star)$, is the strict monoidal $\kk$-linear supercategory generated by a single object $\go$ and the morphisms
	$$\begin{tikzpicture}[anchorbase]
	\draw[-] (0, 0) -- (0.4, 0.4);
	\draw[-] (0.4, 0) -- (0, 0.4);
	\end{tikzpicture} \colon \go \otimes \go \to \go \otimes \go, \quad 
	\begin{tikzpicture}[anchorbase]
	\draw[-] (0, 0) -- (0, 0.2) arc(180:0:0.2) -- (0.4, 0);
	\end{tikzpicture} \colon \go \otimes \go \to \one, \quad
	\begin{tikzpicture}[anchorbase, yscale=-1]
	\draw[-] (0, 0) -- (0, 0.2) arc(180:0:0.2) -- (0.4, 0);
	\end{tikzpicture} \colon \one \to \go \otimes \go, \quad
	\begin{tikzpicture}[anchorbase]
	\draw[-] (0, 0) -- (0, 0.4);
	\token{0, 0.2}{west}{a};
	\end{tikzpicture} \colon \go \to \go \quad  a \in A,$$ 
	subject to the following relations:
	\begin{equation} \label{rel:brauer}
	\begin{tikzpicture}[anchorbase]
	\draw[-] (0, 0) -- (0, 0.2) to[out=90, in=270] (0.4, 0.6) to[out=90, in=270] (0, 1) -- (0, 1.2);
	\draw[-] (0.4, 0) -- (0.4, 0.2) to[out=90, in=270] (0, 0.6) to[out=90, in=270] (0.4, 1) -- (0.4, 1.2);
	\end{tikzpicture}
	\ =\
	\begin{tikzpicture}[anchorbase]
	\draw[-] (0, 0) -- (0, 1.2);
	\draw[-] (0.4, 0) -- (0.4, 1.2);
	\end{tikzpicture}\ ,
	\quad
	\begin{tikzpicture}[anchorbase]
	\draw[-] (0, 0) -- (0, 0.2) to[out=90, in=270] (0.4, 0.6) to[out=90, in=270] (0.8, 1) -- (0.8, 1.6);
	\draw[-] (0.4, 0) -- (0.4, 0.2) to[out=90, in=270] (0, 0.6) -- (0, 1) to[out=90, in=270] (0.4, 1.4) -- (0.4, 1.6);
	\draw[-] (0.8, 0) -- (0.8, 0.6) to[out=90, in=270] (0.4, 1) to[out=90, in=270] (0, 1.4) -- (0, 1.6);
	\end{tikzpicture}
	\ =\
	\begin{tikzpicture}[anchorbase, xscale=-1]
	\draw[-] (0, 0) -- (0, 0.2) to[out=90, in=270] (0.4, 0.6) to[out=90, in=270] (0.8, 1) -- (0.8, 1.6);
	\draw[-] (0.4, 0) -- (0.4, 0.2) to[out=90, in=270] (0, 0.6) -- (0, 1) to[out=90, in=270] (0.4, 1.4) -- (0.4, 1.6);
	\draw[-] (0.8, 0) -- (0.8, 0.6) to[out=90, in=270] (0.4, 1) to[out=90, in=270] (0, 1.4) -- (0, 1.6);
	\end{tikzpicture}\ ,
	\quad
	\begin{tikzpicture}[anchorbase, xscale=-1]
	\draw[-] (0, 0) -- (0, 0.8) arc(180:0:0.2) -- (0.4, 0.6) arc(180:360:0.2) -- (0.8, 1.4);
	\end{tikzpicture}
	\ =\
	\begin{tikzpicture}[anchorbase]
	\draw[-] (0, 0) -- (0, 1.4);
	\end{tikzpicture}
	\ =\
	\begin{tikzpicture}[anchorbase]
	\draw[-] (0, 0) -- (0, 0.8) arc(180:0:0.2) -- (0.4, 0.6) arc(180:360:0.2) -- (0.8, 1.4);
	\end{tikzpicture}\ ,
	\quad
	\begin{tikzpicture}[anchorbase]
	\draw[-] (0, 0) -- (0, 0.2) to[out=90, in=270] (0.4, 0.6) arc(0:180:0.2) to[out=270, in=90] (0.4, 0.2) -- (0.4, 0);
	\end{tikzpicture}
	\ =\
	\begin{tikzpicture}[anchorbase]
	\draw[-] (0, 0) -- (0, 0.2) arc(180:0:0.2) -- (0.4, 0);
	\end{tikzpicture}\ ,
	\quad
	\begin{tikzpicture}[anchorbase]
	\draw[-] (0, 0) -- (0, 0.2) to[out=90, in=270] (0.4, 0.6) arc(180:0:0.2) -- (0.8, 0);
	\draw[-] (0.4, 0) -- (0.4, 0.2) to[out=90, in=270] (0, 0.6) -- (0, 1);
	\end{tikzpicture}
	\ =\
	\begin{tikzpicture}[anchorbase, xscale=-1]
	\draw[-] (0, 0) -- (0, 0.2) to[out=90, in=270] (0.4, 0.6) arc(180:0:0.2) -- (0.8, 0);
	\draw[-] (0.4, 0) -- (0.4, 0.2) to[out=90, in=270] (0, 0.6) -- (0, 1);
	\end{tikzpicture}\ ,
	\end{equation}
	\begin{equation} \label{rel:tokens}
	\begin{tikzpicture}[anchorbase]
	\draw[-] (0, 0) -- (0, 1);
	\token{0, 0.5}{east}{1_A};
	\end{tikzpicture}
	\ =\
	\begin{tikzpicture}[anchorbase]
	\draw[-] (0, 0) -- (0, 1);
	\end{tikzpicture}\ ,
	\quad 
	\lambda\ \begin{tikzpicture}[anchorbase]
	\draw[-] (0, 0) -- (0, 1);
	\token{0, 0.5}{west}{a};
	\end{tikzpicture}
	\ +\
	\mu\ \begin{tikzpicture}[anchorbase]
	\draw[-] (0, 0) -- (0, 1);
	\token{0, 0.5}{west}{b};
	\end{tikzpicture}
	\ =\
	\begin{tikzpicture}[anchorbase]
	\draw[-] (0, 0) -- (0, 1);
	\token{0, 0.5}{west}{\lambda a + \mu b};
	\end{tikzpicture} ,
	\quad 
	\begin{tikzpicture}[anchorbase]
	\draw[-] (0, 0) -- (0, 1);
	\token{0, 0.3}{east}{b};
	\token{0, 0.7}{east}{a};
	\end{tikzpicture}
	\ =\
	\begin{tikzpicture}[anchorbase]
	\draw[-] (0, 0) -- (0, 1);
	\token{0, 0.5}{east}{ab};
	\end{tikzpicture}\ ,
	\quad 
	\begin{tikzpicture}[anchorbase]
	\draw[-] (0, 0) -- (0.8, 0.8);
	\draw[-] (0.8, 0) -- (0, 0.8);
	\token{0.2, 0.2}{east}{a};
	\end{tikzpicture}
	\ =\
	\begin{tikzpicture}[anchorbase]
	\draw[-] (0, 0) -- (0.8, 0.8);
	\draw[-] (0.8, 0) -- (0, 0.8);
	\token{0.6, 0.6}{west}{a};
	\end{tikzpicture}\ ,
	\quad
	\begin{tikzpicture}[anchorbase]
	\draw[-] (0, 0) -- (0, 0.3) arc(180:0:0.2) -- (0.4, 0);
	\token{0, 0.2}{east}{a};
	\end{tikzpicture}
	\ =\
	\begin{tikzpicture}[anchorbase]
	\draw[-] (0, 0) -- (0, 0.3) arc(180:0:0.2) -- (0.4, 0);
	\token{0.4, 0.2}{west}{a^\star};
	\end{tikzpicture},
	\end{equation}
	for all $a, b \in A$ and $\lambda, \mu \in \kk$. The parity of $\begin{tikzpicture}[anchorbase]
	\draw[-] (0, 0) -- (0, 0.4);
	\token{0, 0.2}{west}{a};
	\end{tikzpicture}$ is $\bar{a}$, the morphisms $\begin{tikzpicture}[anchorbase]
	\draw[-] (0, 0) -- (0, 0.2) arc(180:0:0.2) -- (0.4, 0);
	\end{tikzpicture}$,
	$\begin{tikzpicture}[anchorbase, yscale=-1]
	\draw[-] (0, 0) -- (0, 0.2) arc(180:0:0.2) -- (0.4, 0);
	\end{tikzpicture}$, and $\begin{tikzpicture}[anchorbase]
	\draw[-] (0, 0) -- (0.4, 0.4);
	\draw[-] (0.4, 0) -- (0, 0.4);
	\end{tikzpicture}$ are even. The morphisms $\begin{tikzpicture}[anchorbase]
	\draw[-] (0, 0) -- (0, 0.4);
	\token{0, 0.2}{west}{a};
	\end{tikzpicture}$ are called \emph{(Frobenius) tokens}.
	
	For $d \in \kk$, we define $\B(A, -^\star, d)$ to be the quotient of $\B(A, -^\star)$ by the additional relations
	\begin{equation} \label{eq:specialization}
	\begin{tikzpicture}[anchorbase]
	\draw[-] (0, 0) arc(0:360:0.2);
	\token{0, 0}{west}{a};
	\end{tikzpicture}
	\ =\
	d\str_A(a)\id_{\one},
	\end{equation}
	where $a$ ranges over $A$ and $\str_A(a) := \summ_{b \in B_A} (-1)^{\bar{b}}\tr(b^\vee ba)$ is the supertrace of the action of $a$ on $A$. We call $d$ the \emph{specialization parameter}.
\end{defin}

Note that \cite[Def.~9.1]{diagrammaticsRealSupergroups} involves a parameter $\sigma \in \{0, 1\}$, corresponding to the parity of the cup and cap morphisms. The definition of $\B(A, -^\star)$ given above is the case $\sigma = 0$. See Remark \ref{rem:oddCase} for a discussion of why we do not consider the case $\sigma = 1$ in the present paper. Whenever we cite results from \cite{diagrammaticsRealSupergroups}, we always take $\sigma = 0$.

\begin{lem}[{\cite[Prop.~9.3]{diagrammaticsRealSupergroups}}] \label{lem:brauerBasics}
	The following relations hold in $\B(A, -^\star)$ for all $a \in A$:
	\begin{equation} \label{rel:brauerBasics}
	\begin{tikzpicture}[anchorbase, yscale=-1]
	\draw[-] (0, 0) -- (0, 0.2) to[out=90, in=270] (0.4, 0.6) arc(0:180:0.2) to[out=270, in=90] (0.4, 0.2) -- (0.4, 0);
	\end{tikzpicture}
	\ =\
	\begin{tikzpicture}[anchorbase, yscale=-1]
	\draw[-] (0, 0) -- (0, 0.2) arc(180:0:0.2) -- (0.4, 0);
	\end{tikzpicture}\ ,
	\quad \quad
	\begin{tikzpicture}[anchorbase, yscale=-1]
	\draw[-] (0, 0) -- (0, 0.2) to[out=90, in=270] (0.4, 0.6) arc(180:0:0.2) -- (0.8, 0);
	\draw[-] (0.4, 0) -- (0.4, 0.2) to[out=90, in=270] (0, 0.6) -- (0, 1);
	\end{tikzpicture}
	\ =\
	\begin{tikzpicture}[anchorbase, xscale=-1, yscale=-1]
	\draw[-] (0, 0) -- (0, 0.2) to[out=90, in=270] (0.4, 0.6) arc(180:0:0.2) -- (0.8, 0);
	\draw[-] (0.4, 0) -- (0.4, 0.2) to[out=90, in=270] (0, 0.6) -- (0, 1);
	\end{tikzpicture}\ ,
	\quad \quad
	\begin{tikzpicture}[anchorbase]
	\draw[-] (0, 0) -- (0.8, 0.8);
	\draw[-] (0.8, 0) -- (0, 0.8);
	\token{0.2, 0.6}{east}{a};
	\end{tikzpicture}
	\ =\
	\begin{tikzpicture}[anchorbase]
	\draw[-] (0, 0) -- (0.8, 0.8);
	\draw[-] (0.8, 0) -- (0, 0.8);
	\token{0.6, 0.2}{west}{a};
	\end{tikzpicture}\ ,
	\quad
	\begin{tikzpicture}[anchorbase, yscale=-1]
	\draw[-] (0, 0) -- (0, 0.3) arc(180:0:0.2) -- (0.4, 0);
	\token{0, 0.2}{east}{a};
	\end{tikzpicture}
	\ =\
	\begin{tikzpicture}[anchorbase, yscale=-1]
	\draw[-] (0, 0) -- (0, 0.3) arc(180:0:0.2) -- (0.4, 0);
	\token{0.4, 0.2}{west}{a^\star};
	\end{tikzpicture}.
	\end{equation}
\end{lem}

From this point forward, let $V$ be a right $A$-supermodule, $\varphi$ a unimodular $(\nu, -^\star)$-superhermitian form on $V$, and $\Phi = \tr \circ\,\varphi$ the corresponding nondegenerate $(\nu, -^\star)$-supersymmetric form on $V$, as in Lemma \ref{lem:hermitianToSymmetric}. Fix a homogeneous $\kk$-basis $B_V$ of $V$, and let $B_V^\vee = \{b^\vee \mid b \in B_V\}$ be the left dual basis with respect to $\Phi$. Set $\g = \osp(\varphi)$, and fix a homogeneous $\kk$-basis $B_\g$ for $\g$. 

\begin{prop}[{\cite[Thm.~10.1]{diagrammaticsRealSupergroups}}] \label{prop:incarnation}
	There exists a unique monoidal superfunctor, called the \emph{incarnation superfunctor associated to $\Phi$}, denoted $F_\Phi \colon \B(A, -^\star) \to \g\dashsmod$, such that \mbox{$F_\Phi(\go) = V$},
	\begin{align*}
	&F_\Phi\left(\begin{tikzpicture}[anchorbase]
	\draw[-] (0, 0) -- (0.4, 0.4);
	\draw[-] (0.4, 0) -- (0, 0.4);
	\end{tikzpicture} \right) \colon V \otimes V \to V \otimes V, & v \otimes w &\mapsto \nu(-1)^{\bar{v}\bar{w}}w \otimes v,\\
	&F_\Phi\left(\begin{tikzpicture}[anchorbase]
	\draw[-] (0, 0) -- (0, 0.2) arc(180:0:0.2) -- (0.4, 0);
	\end{tikzpicture}\right) \colon V \otimes V \to \kk, & v \otimes w &\mapsto \Phi(v, w),\\
	&F_\Phi\left(\begin{tikzpicture}[anchorbase]
	\draw[-] (0, 0) -- (0, 0.4);
	\token{0, 0.2}{west}{a};
	\end{tikzpicture}\right) \colon V \to V, & v &\mapsto (-1)^{\bar{a}\bar{v}}va^\star.
	\end{align*}
	This superfunctor also satisfies
	\begin{align*}
	F_\Phi\left(\begin{tikzpicture}[anchorbase, yscale=-1]
	\draw[-] (0, 0) -- (0, 0.2) arc(180:0:0.2) -- (0.4, 0);
	\end{tikzpicture}\right) \colon \kk \to V \otimes V, \quad \quad 1 &\mapsto \summ_{v \in B_V} v \otimes v^\vee,
	\end{align*}
	and
	$$F_\Phi\left(\begin{tikzpicture}[anchorbase]
	\draw[-] (0, 0) arc(0:360:0.2);
	\token{0, 0}{west}{a};
	\end{tikzpicture}\right) = \str_V(a) \id_{\one} $$
	for all $a \in A$. Hence in the case $V = A^{m \mid n}$ for some $m, n \in \N$, $F_\Phi$ is also well-defined as a monoidal superfunctor from $\B(A, -^\star, \nu(m - n))$ to $\g\dashsmod$.
\end{prop}

Taking $A = \kk$ in Proposition \ref{prop:incarnation} and restricting our attention to endomorphism algebras, we recover the classic actions of Brauer algebras on modules for orthogonal and symplectic groups (phrased here in terms of the associated Lie algebras). Such actions were first studied by Brauer in \cite{brauerAlgebras}.

In the rest of the paper, we specialize to the case $V = A^{m \mid n}$ for some $m, n \in \N$.

\subsection{Affine Category and Superfunctor} \label{ss:affine}

The definition of affine Frobenius Brauer categories involves \emph{teleporter} morphisms in $\B(A, -^\star)$. The \emph{ordinary teleporter} morphism is defined as follows:
\begin{equation}
\begin{tikzpicture}[anchorbase]
\draw[-] (0, 0) -- (0, 0.8);
\draw[-] (0.6, 0) -- (0.6, 0.8);
\teleportUU{0, 0.6}{0.6, 0.2};
\end{tikzpicture}
\ :=\
\summ_{b \in B_A} \begin{tikzpicture}[anchorbase]
\draw[-] (0, 0) -- (0, 0.8);
\draw[-] (0.6, 0) -- (0.6, 0.8);
\token{0, 0.6}{east}{b};
\token{0.6, 0.2}{west}{b^\vee};
\end{tikzpicture}.
\end{equation}
It is straightforward to show that this definition is independent of the choice of basis $B_A$. The definition does depend on the choice of trace map for $A$, but as we will discuss after Definition \ref{def:affineCategory}, different choices of trace map yield isomorphic affine Frobenius Brauer categories.

One can also draw teleporters with the right endpoint above the left one; the definition is slightly altered such that the dual element appears on the left, i.e.\
\begin{equation}
\begin{tikzpicture}[anchorbase]
\draw[-] (0, 0) -- (0, 0.8);
\draw[-] (0.6, 0) -- (0.6, 0.8);
\teleportUU{0, 0.2}{0.6, 0.6};
\end{tikzpicture}
\ :=\
\summ_{b \in B_A} \begin{tikzpicture}[anchorbase]
\draw[-] (0, 0) -- (0, 0.8);
\draw[-] (0.6, 0) -- (0.6, 0.8);
\token{0, 0.2}{east}{b^\vee};
\token{0.6, 0.6}{west}{b};
\end{tikzpicture}.
\end{equation}
This modification ensures that teleporter endpoints slide up and down freely:
$$\begin{tikzpicture}[anchorbase]
\draw[-] (0, 0) -- (0, 0.8);
\draw[-] (0.6, 0) -- (0.6, 0.8);
\teleportUU{0, 0.6}{0.6, 0.2};
\end{tikzpicture}
=
\summ_{b \in B_A} \begin{tikzpicture}[anchorbase]
\draw[-] (0, 0) -- (0, 0.8);
\draw[-] (0.6, 0) -- (0.6, 0.8);
\token{0, 0.6}{east}{b};
\token{0.6, 0.2}{west}{b^\vee};
\end{tikzpicture} 
\overset{\eqref{eq:superInterchangeLaw}}{=}
\summ_{b \in B_A} (-1)^{\bar{b}} \begin{tikzpicture}[anchorbase]
\draw[-] (0, 0) -- (0, 0.8);
\draw[-] (0.6, 0) -- (0.6, 0.8);
\token{0, 0.2}{east}{b};
\token{0.6, 0.6}{west}{b^\vee};
\end{tikzpicture}
\overset{\eqref{eq:doubleDual}}{=}
\summ_{b \in B_A} \begin{tikzpicture}[anchorbase]
\draw[-] (0, 0) -- (0, 0.8);
\draw[-] (0.6, 0) -- (0.6, 0.8);
\token{0, 0.2}{east}{b^\vee};
\token{0.6, 0.6}{west}{b};
\end{tikzpicture}
=
\begin{tikzpicture}[anchorbase]
\draw[-] (0, 0) -- (0, 0.8);
\draw[-] (0.6, 0) -- (0.6, 0.8);
\teleportUU{0, 0.2}{0.6, 0.6};
\end{tikzpicture},$$
where in the second-to-last equality we switched to a sum over the left dual basis. In light of this sliding identity, we also allow teleporters to be drawn with both endpoints at the same height:
$$\begin{tikzpicture}[anchorbase]
\draw[-] (0, 0) -- (0, 0.8);
\draw[-] (0.6, 0) -- (0.6, 0.8);
\teleportUU{0, 0.4}{0.6, 0.4};
\end{tikzpicture}
\ :=\
\summ_{b \in B_A} \begin{tikzpicture}[anchorbase]
\draw[-] (0, 0) -- (0, 0.8);
\draw[-] (0.6, 0) -- (0.6, 0.8);
\token{0, 0.4}{east}{b};
\token{0.6, 0.4}{west}{b^\vee};
\end{tikzpicture}
\ \overset{\eqref{eq:superInterchangeLaw}}{=}\
\summ_{b \in B_A} \begin{tikzpicture}[anchorbase]
\draw[-] (0, 0) -- (0, 0.8);
\draw[-] (0.6, 0) -- (0.6, 0.8);
\token{0, 0.6}{east}{b};
\token{0.6, 0.2}{west}{b^\vee};
\end{tikzpicture}
\ =\
\begin{tikzpicture}[anchorbase]
\draw[-] (0, 0) -- (0, 0.8);
\draw[-] (0.6, 0) -- (0.6, 0.8);
\teleportUU{0, 0.6}{0.6, 0.2};
\end{tikzpicture}
\ =\
\begin{tikzpicture}[anchorbase]
\draw[-] (0, 0) -- (0, 0.8);
\draw[-] (0.6, 0) -- (0.6, 0.8);
\teleportUU{0, 0.2}{0.6, 0.6};
\end{tikzpicture}.$$
Further, one can draw teleporters with one or both endpoints pointing downwards. The corresponding definitions are similar to ordinary teleporters, but with $-^\star$ applied to the tokens corresponding to the endpoint(s) that are facing downwards. For example,
\begin{align*}
&\begin{tikzpicture}[anchorbase]
\draw[-] (0, 0) -- (0, 0.8);
\draw[-] (0.6, 0) -- (0.6, 0.8);
\teleportUD{0, 0.4}{0.6, 0.4};
\end{tikzpicture}
\ :=\ 
\summ_{b \in B_A}  \begin{tikzpicture}[anchorbase]
\draw[-] (0, 0) -- (0, 0.8);
\draw[-] (0.6, 0) -- (0.6, 0.8);
\token{0, 0.4}{east}{b};
\token{0.6, 0.4}{west}{(b^\vee)^\star};
\end{tikzpicture},
\\
&\begin{tikzpicture}[anchorbase]
\draw[-] (0, 0) -- (0, 0.8);
\draw[-] (0.6, 0) -- (0.6, 0.8);
\teleportDD{0, 0.2}{0.6, 0.6};
\end{tikzpicture}
\ :=\
\summ_{b \in B_A} \begin{tikzpicture}[anchorbase]
\draw[-] (0, 0) -- (0, 0.8);
\draw[-] (0.6, 0) -- (0.6, 0.8);
\token{0, 0.2}{east}{(b^\vee)^\star};
\token{0.6, 0.6}{west}{b^\star};
\end{tikzpicture}
\ \overset{}{=}\
\summ_{b \in B_A} \begin{tikzpicture}[anchorbase]
\draw[-] (0, 0) -- (0, 0.8);
\draw[-] (0.6, 0) -- (0.6, 0.8);
\token{0, 0.2}{east}{b^\vee};
\token{0.6, 0.6}{west}{b};
\end{tikzpicture}
\ =\
\begin{tikzpicture}[anchorbase]
\draw[-] (0, 0) -- (0, 0.8);
\draw[-] (0.6, 0) -- (0.6, 0.8);
\teleportUU{0, 0.2}{0.6, 0.6};
\end{tikzpicture},
\end{align*}
where in the second-to-last equality we switched to a sum over the involuted basis $\{b^\star \mid b \in B_A\}$. In general, this kind of calculation shows that one can flip the orientation of both endpoints of a teleporter simultaneously without changing the morphism, i.e.\ only the relative orientation of the endpoints matters. Teleporters whose endpoints have opposite orientations are called \emph{reflecting teleporters}. The names of the teleporter morphisms are inspired by the identities outlined in Proposition \ref{prop:teleporters}, in which Frobenius tokens ``teleport'' from one string to another.

One can also draw teleporters in larger diagrams. When doing so, one should include the sign $(-1)^{\bar{b}x}$ in the sum defining the teleporter, where $x$ is the sum of the parities of all morphisms in the diagram appearing vertically between the two teleporter endpoints. For instance,
$$\begin{tikzpicture}[anchorbase]
\draw[-] (0, 0) -- (0, 1.2);
\draw[-] (0.6, 0) -- (0.6, 1.2);
\draw[-] (1.2, 0) -- (1.2, 1.2);
\draw[-] (1.8, 0) -- (1.8, 1.2);
\draw[-] (2.4, 0) -- (2.4, 1.2);
\token{0.6, 0.3}{west}{c};
\token{1.2, 0.5}{west}{d};
\token{2.4, 0.7}{west}{e};
\teleportUU{0, 0.15}{1.8, 1.05};
\end{tikzpicture} = \summ_{b \in B_A} (-1)^{\bar{b}(\bar{c} + \bar{d} + \bar{e})}\begin{tikzpicture}[anchorbase]
\draw[-] (0, 0) -- (0, 1.2);
\draw[-] (0.6, 0) -- (0.6, 1.2);
\draw[-] (1.2, 0) -- (1.2, 1.2);
\draw[-] (1.8, 0) -- (1.8, 1.2);
\draw[-] (2.4, 0) -- (2.4, 1.2);
\token{0.6, 0.3}{west}{c};
\token{1.2, 0.5}{west}{d};
\token{2.4, 0.7}{west}{e};
\token{0, 0.15}{east}{b^\vee};
\token{1.8, 1.05}{west}{b};
\end{tikzpicture}.$$
This convention ensures that one can freely slide the endpoints of teleporters along strings; the signs arising from \eqref{eq:superInterchangeLaw} do not need to be actively tracked since they are incorporated into the definition of the teleporters.
For instance, we have:
$$
\begin{tikzpicture}[anchorbase]
\draw[-] (0, 0) -- (0, 1.2);
\draw[-] (0.6, 0) -- (0.6, 1.2);
\draw[-] (1.2, 0) -- (1.2, 1.2);
\draw[-] (1.8, 0) -- (1.8, 1.2);
\draw[-] (2.4, 0) -- (2.4, 1.2);
\token{0.6, 0.3}{west}{c};
\token{1.2, 0.5}{west}{d};
\token{2.4, 0.7}{west}{e};
\teleportUU{0, 0.15}{1.8, 1.05};
\end{tikzpicture}
=
\begin{tikzpicture}[anchorbase]
\draw[-] (0, 0) -- (0, 1.2);
\draw[-] (0.6, 0) -- (0.6, 1.2);
\draw[-] (1.2, 0) -- (1.2, 1.2);
\draw[-] (1.8, 0) -- (1.8, 1.2);
\draw[-] (2.4, 0) -- (2.4, 1.2);
\token{0.6, 0.3}{west}{c};
\token{1.2, 0.5}{west}{d};
\token{2.4, 0.7}{west}{e};
\teleportUU{0, 1.05}{1.8, 1.05};
\end{tikzpicture}
=
\begin{tikzpicture}[anchorbase]
\draw[-] (0, 0) -- (0, 1.2);
\draw[-] (0.6, 0) -- (0.6, 1.2);
\draw[-] (1.2, 0) -- (1.2, 1.2);
\draw[-] (1.8, 0) -- (1.8, 1.2);
\draw[-] (2.4, 0) -- (2.4, 1.2);
\token{0.6, 0.3}{west}{c};
\token{1.2, 0.5}{west}{d};
\token{2.4, 0.7}{west}{e};
\teleportUU{0, 0.15}{1.8, 0.15};
\end{tikzpicture}
=
\begin{tikzpicture}[anchorbase]
\draw[-] (0, 0) -- (0, 1.2);
\draw[-] (0.6, 0) -- (0.6, 1.2);
\draw[-] (1.2, 0) -- (1.2, 1.2);
\draw[-] (1.8, 0) -- (1.8, 1.2);
\draw[-] (2.4, 0) -- (2.4, 1.2);
\token{0.6, 0.3}{west}{c};
\token{1.2, 0.6}{west}{d};
\token{2.4, 0.8}{west}{e};
\teleportUU{0, 1.05}{1.8, 0.15};
\end{tikzpicture}.
$$

\begin{prop} \label{prop:teleporters}
	Teleporter morphisms satisfy the following relations, for all $a \in A$:
	\begin{equation} \label{rel:teleporting}
	\begin{tikzpicture}[anchorbase]
	\draw[-] (0, 0) -- (0, 0.8);
	\draw[-] (0.6, 0) -- (0.6, 0.8);
	\teleportUU{0, 0.4}{0.6, 0.4};
	\token{0, 0.6}{east}{a};
	\end{tikzpicture}
	\ =\
	\begin{tikzpicture}[anchorbase]
	\draw[-] (0, 0) -- (0, 0.8);
	\draw[-] (0.6, 0) -- (0.6, 0.8);
	\teleportUU{0, 0.4}{0.6, 0.4};
	\token{0.6, 0.2}{west}{a};
	\end{tikzpicture},
	\quad \quad
	\begin{tikzpicture}[anchorbase]
	\draw[-] (0, 0) -- (0, 0.8);
	\draw[-] (0.6, 0) -- (0.6, 0.8);
	\teleportUU{0, 0.4}{0.6, 0.4};
	\token{0.6, 0.6}{west}{a};
	\end{tikzpicture}
	\ =\
	\begin{tikzpicture}[anchorbase]
	\draw[-] (0, 0) -- (0, 0.8);
	\draw[-] (0.6, 0) -- (0.6, 0.8);
	\teleportUU{0, 0.4}{0.6, 0.4};
	\token{0, 0.2}{east}{a};
	\end{tikzpicture}\ ,
	\end{equation}
	\begin{equation}
	\begin{tikzpicture}[anchorbase]
	\draw[-] (0, 0) -- (0, 0.8);
	\draw[-] (0.6, 0) -- (0.6, 0.8);
	\teleportUD{0, 0.4}{0.6, 0.4};
	\token{0, 0.6}{east}{a};
	\end{tikzpicture}
	\ =\
	\begin{tikzpicture}[anchorbase]
	\draw[-] (0, 0) -- (0, 0.8);
	\draw[-] (0.6, 0) -- (0.6, 0.8);
	\teleportUD{0, 0.4}{0.6, 0.4};
	\token{0.6, 0.6}{west}{a^\star};
	\end{tikzpicture},
	\quad \quad
	\begin{tikzpicture}[anchorbase]
	\draw[-] (0, 0) -- (0, 0.8);
	\draw[-] (0.6, 0) -- (0.6, 0.8);
	\teleportUD{0, 0.4}{0.6, 0.4};
	\token{0.6, 0.6}{west}{a};
	\end{tikzpicture}
	\ =\
	\begin{tikzpicture}[anchorbase]
	\draw[-] (0, 0) -- (0, 0.8);
	\draw[-] (0.6, 0) -- (0.6, 0.8);
	\teleportUD{0, 0.4}{0.6, 0.4};
	\token{0, 0.6}{east}{a^\star};
	\end{tikzpicture}\ .
	\end{equation}
	(Note that tokens travel from one vertical side of an ordinary teleporter to the other, but stay on the same vertical side after travelling through a reflecting teleporter.)
\end{prop}

\begin{proof}
	We will prove the first identity; the others are similar.
	\begin{multline*}
	\begin{tikzpicture}[anchorbase]
	\draw[-] (0, 0) -- (0, 0.8);
	\draw[-] (0.6, 0) -- (0.6, 0.8);
	\teleportUU{0, 0.4}{0.6, 0.4};
	\token{0, 0.6}{east}{a};
	\end{tikzpicture}
	\overset{\eqref{rel:tokens}}{=} 
	\summ_{b \in B_A} 
	\begin{tikzpicture}[anchorbase]
	\draw[-] (0, 0) -- (0, 0.8);
	\draw[-] (0.6, 0) -- (0.6, 0.8);
	\token{0, 0.6}{east}{ab};
	\token{0.6, 0.2}{west}{b^\vee};
	\end{tikzpicture}
	\underset{\eqref{eq:doubleDual}}{\overset{\eqref{eq:traceDecomposition}}{=}}
	\summ_{b \in B_A} \summ_{c \in B_A} \begin{tikzpicture}[anchorbase]
	\draw[-] (0, 0) -- (0, 0.8);
	\draw[-] (0.6, 0) -- (0.6, 0.8);
	\token{0, 0.6}{east}{\tr(c^\vee ab)c};
	\token{0.6, 0.2}{west}{b^\vee};
	\end{tikzpicture}
	\\
	\overset{\eqref{rel:tokens}}{=}
	\summ_{c \in B_A} \summ_{b \in B_A}  \begin{tikzpicture}[anchorbase]
	\draw[-] (0, 0) -- (0, 0.8);
	\draw[-] (0.6, 0) -- (0.6, 0.8);
	\token{0, 0.6}{east}{c};
	\token{0.6, 0.2}{west}{\tr(c^\vee a b)b^\vee};
	\end{tikzpicture}
	\overset{\eqref{eq:traceDecomposition}}{=}
	\summ_{c \in B_A}  \begin{tikzpicture}[anchorbase]
	\draw[-] (0, 0) -- (0, 0.8);
	\draw[-] (0.6, 0) -- (0.6, 0.8);
	\token{0, 0.6}{east}{c};
	\token{0.6, 0.2}{west}{c^\vee a};
	\end{tikzpicture}
	\overset{\eqref{rel:tokens}}{=}
	\begin{tikzpicture}[anchorbase]
	\draw[-] (0, 0) -- (0, 0.8);
	\draw[-] (0.6, 0) -- (0.6, 0.8);
	\teleportUU{0, 0.4}{0.6, 0.4};
	\token{0.6, 0.2}{east}{a};
	\end{tikzpicture}\ . \qedhere
	\end{multline*}
\end{proof}

\begin{lem} \label{lem:teleporterCrossingSlide}
	Teleporters of both types slide through crossings:
	\begin{equation} \label{rel:teleporterCrossingSlide}
	\begin{tikzpicture}[anchorbase]
	\draw[-] (0, 0) -- (0.8, 0.8);
	\draw[-] (0.8, 0) -- (0, 0.8);
	\teleportUU{0.2, 0.2}{0.6, 0.2};
	\end{tikzpicture}
	\ =\
	\begin{tikzpicture}[anchorbase]
	\draw[-] (0, 0) -- (0.8, 0.8);
	\draw[-] (0.8, 0) -- (0, 0.8);
	\teleportUU{0.2, 0.6}{0.6, 0.6};
	\end{tikzpicture},
	\quad \quad
	\begin{tikzpicture}[anchorbase]
	\draw[-] (0, 0) -- (0.8, 0.8);
	\draw[-] (0.8, 0) -- (0, 0.8);
	\teleportUD{0.2, 0.2}{0.6, 0.2};
	\end{tikzpicture}
	\ =\
	\begin{tikzpicture}[anchorbase]
	\draw[-] (0, 0) -- (0.8, 0.8);
	\draw[-] (0.8, 0) -- (0, 0.8);
	\teleportDU{0.2, 0.6}{0.6, 0.6};
	\end{tikzpicture}.
	\end{equation}
\end{lem}

\begin{proof}
	We have:
	\begin{multline*}
	\begin{tikzpicture}[anchorbase]
	\draw[-] (0, 0) -- (0.8, 0.8);
	\draw[-] (0.8, 0) -- (0, 0.8);
	\teleportUU{0.2, 0.2}{0.6, 0.2};
	\end{tikzpicture}
	=
	\summ_{b \in B_A}  \begin{tikzpicture}[anchorbase]
	\draw[-] (0, 0) -- (0.8, 0.8);
	\draw[-] (0.8, 0) -- (0, 0.8);
	\token{0.2, 0.2}{east}{b};
	\token{0.6, 0.2}{west}{b^\vee};
	\end{tikzpicture}
	=
	\summ_{b \in B_A} \begin{tikzpicture}[anchorbase]
	\draw[-] (0, 0) -- (0.8, 0.8);
	\draw[-] (0.8, 0) -- (0, 0.8);
	\token{0.6, 0.6}{west}{b};
	\token{0.6, 0.2}{west}{b^\vee};
	\end{tikzpicture}
	\overset{\eqref{eq:superInterchangeLaw}}{=}
	\summ_{b \in B_A} (-1)^{\bar{b}} \begin{tikzpicture}[anchorbase]
	\draw[-] (0, 0) -- (0.8, 0.8);
	\draw[-] (0.8, 0) -- (0, 0.8);
	\token{0.6, 0.6}{west}{b};
	\token{0.2, 0.6}{east}{b^\vee};
	\end{tikzpicture}
	\overset{\eqref{eq:doubleDual}}{=}
	\summ_{b \in B_A}  \begin{tikzpicture}[anchorbase]
	\draw[-] (0, 0) -- (0.8, 0.8);
	\draw[-] (0.8, 0) -- (0, 0.8);
	\token{0.6, 0.6}{west}{b^\vee};
	\token{0.2, 0.6}{east}{b};
	\end{tikzpicture}
	= \begin{tikzpicture}[anchorbase]
	\draw[-] (0, 0) -- (0.8, 0.8);
	\draw[-] (0.8, 0) -- (0, 0.8);
	\teleportUU{0.2, 0.6}{0.6, 0.6};
	\end{tikzpicture}.
	\end{multline*}
	Note that we changed to a sum over the dual basis in the second-to-last equality. The calculation for reflecting teleporters is essentially the same.
\end{proof}

\begin{lem} \label{lem:teleporterCupCapSliding}
	Teleporter endpoints slide across cups and caps, flipping orientation in the process, e.g.\ ${\begin{tikzpicture}[anchorbase]
		\draw[-] (0, 0) -- (0, 0.5) arc(180:0:0.25) -- (0.5, 0);
		\draw[-] (1, 0) -- (1, 0.75);	
		\teleportUU{0.5, 0.4}{1, 0.15};
		\end{tikzpicture}
		= \begin{tikzpicture}[anchorbase]
		\draw[-] (0, 0) -- (0, 0.5) arc(180:0:0.25) -- (0.5, 0);
		\draw[-] (1, 0) -- (1, 0.75);	
		\teleportDU{0, 0.4}{1, 0.15};
		\end{tikzpicture}} \text{ and } {\begin{tikzpicture}[anchorbase]
		\draw[-] (0, 0) -- (0, -0.5) arc(180:360:0.25) -- (0.5, 0);
		\draw[-] (1, 0) -- (1, -0.75);
		\teleportUU{0.5, -0.45}{1, -0.2};
		\end{tikzpicture} = \begin{tikzpicture}[anchorbase]
		\draw[-] (0, 0) -- (0, -0.5) arc(180:360:0.25) -- (0.5, 0);
		\draw[-] (1, 0) -- (1, -0.75);
		\teleportDU{0, -0.45}{1, -0.2};
		\end{tikzpicture}}.$
\end{lem}

\begin{proof}
	We have:
	\begin{equation*}
	\begin{tikzpicture}[anchorbase]
	\draw[-] (0, 0) -- (0, 0.5) arc(180:0:0.25) -- (0.5, 0);
	\draw[-] (1, 0) -- (1, 0.75);	
	\teleportUU{0.5, 0.4}{1, 0.15};
	\end{tikzpicture}
	\ =\
	\summ_{b \in B_A}
	\begin{tikzpicture}[anchorbase]
	\draw[-] (0, 0) -- (0, 0.5) arc(180:0:0.25) -- (0.5, 0);
	\draw[-] (1, 0) -- (1, 0.75);
	\token{0.5, 0.4}{east}{b};
	\token{1, 0.15}{west}{b^\vee};
	\end{tikzpicture}
	\overset{\eqref{rel:tokens}}{=}
	\summ_{b \in B_A}
	\begin{tikzpicture}[anchorbase]
	\draw[-] (0, 0) -- (0, 0.5) arc(180:0:0.25) -- (0.5, 0);
	\draw[-] (1, 0) -- (1, 0.75);
	\token{0, 0.4}{east}{b^\star};
	\token{1, 0.15}{west}{b^\vee};
	\end{tikzpicture}
	\ =\
	\begin{tikzpicture}[anchorbase]
	\draw[-] (0, 0) -- (0, 0.5) arc(180:0:0.25) -- (0.5, 0);
	\draw[-] (1, 0) -- (1, 0.75);	
	\teleportDU{0, 0.4}{1, 0.15};
	\end{tikzpicture}.
	\end{equation*}
	The second example from the lemma statement follows from an analogous calculation, using \eqref{rel:brauerBasics} in place of \eqref{rel:tokens}.
\end{proof}

Lemma \ref{lem:teleporterCupCapSliding} allows us to unambiguously draw diagrams with sideways teleporter endpoints appearing at critical points of cups and caps; we define such diagrams to be equal to the morphism obtained by sliding the endpoints to either side of the cup(s) and/or cap(s). For instance,
$$\begin{tikzpicture}[anchorbase]
\draw[-] (0, 0) -- (0, 0.3) arc(180:0:0.3) -- (0.6, 0);
\draw[-] (0, 1.5) -- (0, 1.2) arc(180:360:0.3) -- (0.6, 1.5);
\teleportRR{0.3, 0.6}{0.3, 0.9};
\end{tikzpicture}
= \begin{tikzpicture}[anchorbase]
\draw[-] (0, 0) -- (0, 0.3) arc(180:0:0.3) -- (0.6, 0);
\draw[-] (0, 1.5) -- (0, 1.2) arc(180:360:0.3) -- (0.6, 1.5);
\teleportUD{0, 0.2}{0, 1.3};
\end{tikzpicture}
= \begin{tikzpicture}[anchorbase]
\draw[-] (0, 0) -- (0, 0.3) arc(180:0:0.3) -- (0.6, 0);
\draw[-] (0, 1.5) -- (0, 1.2) arc(180:360:0.3) -- (0.6, 1.5);
\teleportUD{0.6, 1.3}{0.6, 0.2};
\end{tikzpicture}
= \begin{tikzpicture}[anchorbase]
\draw[-] (0, 0) -- (0, 0.3) arc(180:0:0.3) -- (0.6, 0);
\draw[-] (0, 1.5) -- (0, 1.2) arc(180:360:0.3) -- (0.6, 1.5);
\teleportUU{0, 0.2}{0.6, 1.3};
\end{tikzpicture}
= \begin{tikzpicture}[anchorbase]
\draw[-] (0, 0) -- (0, 0.3) arc(180:0:0.3) -- (0.6, 0);
\draw[-] (0, 1.5) -- (0, 1.2) arc(180:360:0.3) -- (0.6, 1.5);
\teleportDD{0.6, 0.2}{0, 1.3};
\end{tikzpicture}\ .$$

\begin{defin} \label{def:affineCategory}
	The \emph{(unoriented) affine Frobenius Brauer category associated to $(A, \tr, -^\star)$,} denoted $\AB(A, -^\star)$, is the supercategory obtained from $\B(A, -^\star)$ by adjoining one new even generating morphism,
	$\begin{tikzpicture}[anchorbase]
	\draw[-] (0, 0) -- (0, 0.4);
	\adot{(0, 0.2)};
	\end{tikzpicture} \colon \go \to \go,$
	called an \emph{(affine) dot}, subject to the following relations:
	\begin{equation} \label{rel:dotCrossingSlide}
	\begin{tikzpicture}[anchorbase]
	\draw[-] (0, 0) -- (0.8, 0.8);
	\draw[-] (0.8, 0) -- (0, 0.8);
	\adot{(0.2, 0.6)};
	\end{tikzpicture}
	\ -\
	\begin{tikzpicture}[anchorbase]
	\draw[-] (0, 0) -- (0.8, 0.8);
	\draw[-] (0.8, 0) -- (0, 0.8);
	\adot{(0.6, 0.2)};
	\end{tikzpicture}
	\ =\
	\begin{tikzpicture}[anchorbase]
	\draw[-] (0, 0) -- (0, 0.8);
	\draw[-] (0.6, 0) -- (0.6, 0.8);
	\teleportUU{0, 0.4}{0.6, 0.4};
	\end{tikzpicture}
	\ -\
	\begin{tikzpicture}[anchorbase]
	\draw[-] (0, 0) -- (0, 0.3) arc(180:0:0.3) -- (0.6, 0);
	\draw[-] (0, 1.5) -- (0, 1.2) arc(180:360:0.3) -- (0.6, 1.5);
	\teleportRR{0.3, 0.6}{0.3, 0.9};
	\end{tikzpicture}\ ,
	\end{equation}
	\begin{equation} \label{rel:dotCapSlide}
	\begin{tikzpicture}[anchorbase]
	\draw[-] (0, 0) -- (0, 0.3) arc(180:0:0.3) -- (0.6, 0);
	\adot{(0, 0.2)}; 
	\end{tikzpicture}
	\ =\
	-\ 
	\begin{tikzpicture}[anchorbase]
	\draw[-] (0, 0) -- (0, 0.3) arc(180:0:0.3) -- (0.6, 0);
	\adot{(0.6, 0.2)}; 
	\end{tikzpicture}\ ,
	\end{equation}
	\begin{equation} \label{rel:dotTokenCommuting}
	\begin{tikzpicture}[anchorbase]
	\draw[-] (0, 0) -- (0, 0.8);
	\adot{(0, 0.25)};
	\token{0, 0.55}{west}{a};
	\end{tikzpicture}
	\ =\
	\begin{tikzpicture}[anchorbase]
	\draw[-] (0, 0) -- (0, 0.8);
	\adot{(0, 0.55)};
	\token{0, 0.25}{west}{a};
	\end{tikzpicture} \text{ for all } a \in A.
	\end{equation}
	As with the non-affine category, for any $d \in \kk$, we define $\AB(A, -^\star, d)$ to be the quotient of $\AB(A, -^\star)$ by the additional relations \eqref{eq:specialization}.
\end{defin}

As mentioned previously, teleporter morphisms depend on the choice of trace map for $A$, so to be fully precise one needs to indicate which trace map is being used to define $\AB(A, -^\star)$. However, the following result shows that different choices of trace map yield isomorphic categories. As such, we choose to suppress this detail in our notation for the affine Frobenius Brauer category for the sake of simplicity.

\begin{prop} \label{prop:traceIndependence}
	Let $A$ be a superalgebra, and suppose that $(A, \tr_1, -^\star)$ and $(A, \tr_2, -^\star)$ are symmetric involutive Frobenius superalgebras, using the same involution $-^\star$ in both cases. Write $\AB(A, -^\star)_i$ for the affine Frobenius Brauer category defined with respect to $\tr_i$. There exists an even invertible element $u \in A$ and an isomorphism $T \colon \AB(A, -^\star)_1 \to \AB(A, -^\star)_2$ given by:
	$$T(\go) = \go, \quad T(f) = f \text{ for all } f \in \{\begin{tikzpicture}[anchorbase]
	\draw[-] (0, 0) -- (0.4, 0.4);
	\draw[-] (0.4, 0) -- (0, 0.4);
	\end{tikzpicture}, \begin{tikzpicture}[anchorbase]
	\draw[-] (0, 0) -- (0, 0.2) arc(180:0:0.2) -- (0.4, 0);
	\end{tikzpicture}, \begin{tikzpicture}[anchorbase, yscale=-1]
	\draw[-] (0, 0) -- (0, 0.2) arc(180:0:0.2) -- (0.4, 0);
	\end{tikzpicture}, \begin{tikzpicture}[anchorbase]
	\draw[-] (0, 0) -- (0, 0.4);
	\token{0, 0.2}{west}{a};
	\end{tikzpicture} : a \in A\}, \quad T\left(\begin{tikzpicture}[anchorbase]
	\draw[-] (0, 0) -- (0, 0.4);
	\adot{(0, 0.2)};
	\end{tikzpicture} \right) = \begin{tikzpicture}[anchorbase]
	\draw[-] (0, 0) -- (0, 0.8);
	\adot{(0, 0.25)};
	\token{0, 0.55}{west}{u};
	\end{tikzpicture}.$$
\end{prop}
\begin{proof}	
	Since $\tr_1$ and $\tr_2$ are both trace maps for $A$, there exists an invertible element $u \in A$ such that $\tr_2(a) = \tr_1(ua)$ for all $a \in A$; see e.g.\ \cite[Prop.~2.1.6]{abrams_1997} for a proof. Moreover, since both $\tr_1$ and $\tr_2$ are even, symmetric, and compatible with $-^\star$, we get that $u$ is even, central, and satisfies $u^\star = u$. Writing $-^{\vee_i}$ for left duals taken with respect to $\tr_i$, we have $b^{\vee_2} = u^{-1}b^{\vee_1}$. Using subscripts in the same way for teleporters, this yields:
	\begin{equation} \label{eq:teleporterConversion}
	\begin{tikzpicture}[anchorbase]
	\draw[-] (0, 0) -- (0, 0.8);
	\draw[-] (0.6, 0) -- (0.6, 0.8);
	\teleportUU{0, 0.4}{0.6, 0.4};
	\node at (0.8, 0) {$\scriptstyle 2$};
	\end{tikzpicture}
	=
	\summ_{b \in B_A} \begin{tikzpicture}[anchorbase]
	\draw[-] (0, 0) -- (0, 0.8);
	\draw[-] (0.6, 0) -- (0.6, 0.8);
	\token{0, 0.4}{east}{b};
	\token{0.6, 0.4}{west}{b^{\vee_2}};
	\end{tikzpicture}
	=
	\summ_{b \in B_A} \begin{tikzpicture}[anchorbase]
	\draw[-] (0, 0) -- (0, 0.8);
	\draw[-] (0.6, 0) -- (0.6, 0.8);
	\token{0, 0.4}{east}{b};
	\token{0.6, 0.4}{west}{u^{-1}b^{\vee_1}};
	\end{tikzpicture}
	\overset{\eqref{rel:tokens}}{=}
	\summ_{b \in B_A} \begin{tikzpicture}[anchorbase]
	\draw[-] (0, 0) -- (0, 0.8);
	\draw[-] (0.6, 0) -- (0.6, 0.8);
	\token{0, 0.2}{east}{b};
	\token{0.6, 0.2}{west}{b^{\vee_1}};
	\token{0.6, 0.6}{west}{u^{-1}};
	\end{tikzpicture}
	=
	\begin{tikzpicture}[anchorbase]
	\draw[-] (0, 0) -- (0, 0.8);
	\draw[-] (0.6, 0) -- (0.6, 0.8);
	\teleportUU{0, 0.2}{0.6, 0.2};
	\token{0.6, 0.6}{west}{u^{-1}};
	\node at (0.8, 0) {$\scriptstyle 1$};
	\end{tikzpicture}\ .
	\end{equation}
	It is immediate that $T$ respects all of the defining relations for $\B(A, -^\star)$. For \eqref{rel:dotCrossingSlide}, we compute:
	\begin{multline*}
		T\left(\begin{tikzpicture}[anchorbase]
		\draw[-] (0, 0) -- (0.8, 0.8);
		\draw[-] (0.8, 0) -- (0, 0.8);
		\adot{(0.2, 0.6)};
		\end{tikzpicture}
		\ -\
		\begin{tikzpicture}[anchorbase]
		\draw[-] (0, 0) -- (0.8, 0.8);
		\draw[-] (0.8, 0) -- (0, 0.8);
		\adot{(0.6, 0.2)};
		\end{tikzpicture} \right) 
		= \begin{tikzpicture}[anchorbase]
		\draw[-] (0, 0) -- (0.8, 0.8);
		\draw[-] (0.8, 0) -- (0, 0.8);
		\adot{(0.275, 0.525)};
		\token{0.125, 0.675}{east}{u};
		\end{tikzpicture}
		\ -\
		\begin{tikzpicture}[anchorbase]
		\draw[-] (0, 0) -- (0.8, 0.8);
		\draw[-] (0.8, 0) -- (0, 0.8);
		\adot{(0.675, 0.125)};
		\token{0.525, 0.275}{west}{\vphantom{\frac{a}{b}}u};
		\end{tikzpicture}
		\overset{\eqref{rel:brauerBasics}}{=} \begin{tikzpicture}[anchorbase]
		\draw[-] (0, 0) -- (0.8, 0.8);
		\draw[-] (0.8, 0) -- (0, 0.8);
		\adot{(0.275, 0.525)};
		\token{0.125, 0.675}{east}{u};
		\end{tikzpicture}
		\ -\
		\begin{tikzpicture}[anchorbase]
		\draw[-] (0, 0) -- (0.8, 0.8);
		\draw[-] (0.8, 0) -- (0, 0.8);
		\adot{(0.6, 0.2)};
		\token{0.2, 0.6}{east}{u};
		\end{tikzpicture}
		\overset{\eqref{rel:dotCrossingSlide}}{=}
		\begin{tikzpicture}[anchorbase]
		\draw[-] (0, 0) -- (0, 0.8);
		\draw[-] (0.6, 0) -- (0.6, 0.8);
		\teleportUU{0, 0.2}{0.6, 0.2};
		\token{0, 0.6}{east}{u};
		\node at (0.8, 0) {$\scriptstyle 2$};
		\end{tikzpicture}
		\ -\
		\begin{tikzpicture}[anchorbase]
		\draw[-] (0, 0) -- (0, 0.3) arc(180:0:0.3) -- (0.6, 0);
		\draw[-] (0, 1.5) -- (0, 1.2) arc(180:360:0.3) -- (0.6, 1.5);
		\teleportRR{0.3, 0.6}{0.3, 0.9};
		\token{0, 1.2}{east}{u};
		\node at (0.8, 0) {$\scriptstyle 2$};
		\end{tikzpicture}\\
		\overset{\eqref{eq:teleporterConversion}}{=}
		\begin{tikzpicture}[anchorbase]
		\draw[-] (0, 0) -- (0, 0.8);
		\draw[-] (0.6, 0) -- (0.6, 0.8);
		\teleportUU{0, 0.2}{0.6, 0.2};
		\token{0, 0.6}{east}{u};
		\token{0.6, 0.6}{west}{u^{-1}};
		\node at (0.8, 0) {$\scriptstyle 1$};
		\end{tikzpicture}
		\ -\
		\begin{tikzpicture}[anchorbase]
		\draw[-] (0, 0) -- (0, 0.3) arc(180:0:0.3) -- (0.6, 0);
		\draw[-] (0, 1.5) -- (0, 1.2) arc(180:360:0.3) -- (0.6, 1.5);
		\teleportRR{0.3, 0.6}{0.3, 0.9};
		\token{0.6, 0.3}{west}{u^{-1}};
		\token{0, 1.2}{east}{u};
		\node at (0.8, 0) {$\scriptstyle 1$};
		\end{tikzpicture}
		\overset{\eqref{rel:teleporting}}{=}
		\begin{tikzpicture}[anchorbase]
		\draw[-] (0, 0) -- (0, 1);
		\draw[-] (0.6, 0) -- (0.6, 1);
		\teleportUU{0, 0.5}{0.6, 0.5};
		\token{0, 0.2}{east}{u^{-1}};
		\token{0, 0.8}{east}{u};
		\node at (0.8, 0) {$\scriptstyle 1$};
		\end{tikzpicture}
		\ -\
		\begin{tikzpicture}[anchorbase]
		\draw[-] (0, 0) -- (0, 0.3) arc(180:0:0.3) -- (0.6, 0);
		\draw[-] (0, 1.7) -- (0, 1.2) arc(180:360:0.3) -- (0.6, 1.7);
		\teleportRR{0.3, 0.6}{0.3, 0.9};
		\token{0, 1.2}{east}{u^{-1}};
		\token{0, 1.5}{east}{u};
		\node at (0.8, 0) {$\scriptstyle 1$};
		\end{tikzpicture}\\
		\overset{*}{=}
		\begin{tikzpicture}[anchorbase]
		\draw[-] (0, 0) -- (0, 1);
		\draw[-] (0.6, 0) -- (0.6, 1);
		\teleportUU{0, 0.2}{0.6, 0.2};
		\token{0, 0.5}{east}{u^{-1}};
		\token{0, 0.8}{east}{u};
		\node at (0.8, 0) {$\scriptstyle 1$};
		\end{tikzpicture}
		\ -\
		\begin{tikzpicture}[anchorbase]
		\draw[-] (0, 0) -- (0, 0.3) arc(180:0:0.3) -- (0.6, 0);
		\draw[-] (0, 1.7) -- (0, 1.2) arc(180:360:0.3) -- (0.6, 1.7);
		\teleportRR{0.3, 0.6}{0.3, 0.9};
		\token{0, 1.2}{east}{u^{-1}};
		\token{0, 1.5}{east}{u};
		\node at (0.8, 0) {$\scriptstyle 1$};
		\end{tikzpicture}
		\overset{\eqref{rel:tokens}}{=}
		\begin{tikzpicture}[anchorbase]
		\draw[-] (0, 0) -- (0, 0.8);
		\draw[-] (0.6, 0) -- (0.6, 0.8);
		\teleportUU{0, 0.4}{0.6, 0.4};
		\node at (0.8, 0) {$\scriptstyle 1$};
		\end{tikzpicture}
		\ -\
		\begin{tikzpicture}[anchorbase]
		\draw[-] (0, 0) -- (0, 0.3) arc(180:0:0.3) -- (0.6, 0);
		\draw[-] (0, 1.5) -- (0, 1.2) arc(180:360:0.3) -- (0.6, 1.5);
		\teleportRR{0.3, 0.6}{0.3, 0.9};
		\node at (0.8, 0) {$\scriptstyle 1$};
		\end{tikzpicture}\\
		= T\left(
		\begin{tikzpicture}[anchorbase]
		\draw[-] (0, 0) -- (0, 0.8);
		\draw[-] (0.6, 0) -- (0.6, 0.8);
		\teleportUU{0, 0.4}{0.6, 0.4};
		\node at (0.8, 0) {$\scriptstyle 1$};
		\end{tikzpicture}
		\ -\
		\begin{tikzpicture}[anchorbase]
		\draw[-] (0, 0) -- (0, 0.3) arc(180:0:0.3) -- (0.6, 0);
		\draw[-] (0, 1.5) -- (0, 1.2) arc(180:360:0.3) -- (0.6, 1.5);
		\teleportRR{0.3, 0.6}{0.3, 0.9};
		\node at (0.8, 0) {$\scriptstyle 1$};
		\end{tikzpicture}\right),
	\end{multline*}
	where the equality labelled $*$ follows from the fact that $u$ (and hence $u^{-1}$) is even and central, and so $\begin{tikzpicture}[anchorbase]
	\draw[-] (0, 0) -- (0, 0.4);
	\token{0, 0.2}{west}{u^{-1}};
	\end{tikzpicture}$ slides through teleporter endpoints. A quick calculation shows that $T$ respects \eqref{rel:dotCapSlide} and \eqref{rel:dotTokenCommuting}, using those two relations in the target category together with \eqref{rel:tokens} and the fact that $u$ is central and satisfies $u^\star = u$. This shows $T$ is well-defined. Noting that $\tr_1(a) = \tr_2(u^{-1}a)$ for all $a \in A$, reversing the roles of $\tr_1$ and $\tr_2$ yields a functor $\AB(A, -^\star)_2 \to \AB(A, -^\star)_1$ sending $\begin{tikzpicture}[anchorbase]
	\draw[-] (0, 0) -- (0, 0.4);
	\adot{(0, 0.2)};
	\end{tikzpicture}$ to $\begin{tikzpicture}[anchorbase]
	\draw[-] (0, 0) -- (0, 0.8);
	\adot{(0, 0.25)};
	\token{0, 0.55}{west}{u^{-1}};
	\end{tikzpicture}$, which is the inverse of $T$. Hence $T$ is an isomorphism.
\end{proof}

\begin{lem} \label{lem:affineBasics}
	The following relations hold in $\AB(A, -^\star)$:
	\begin{equation} \label{rel:mirroredDotSlide}
	\begin{tikzpicture}[anchorbase]
	\draw[-] (0, 0) -- (0.8, 0.8);
	\draw[-] (0.8, 0) -- (0, 0.8);
	\adot{(0.2, 0.2)};
	\end{tikzpicture}
	\ -\
	\begin{tikzpicture}[anchorbase]
	\draw[-] (0, 0) -- (0.8, 0.8);
	\draw[-] (0.8, 0) -- (0, 0.8);
	\adot{(0.6, 0.6)};
	\end{tikzpicture}
	\ =\
	\begin{tikzpicture}[anchorbase]
	\draw[-] (0, 0) -- (0, 0.8);
	\draw[-] (0.6, 0) -- (0.6, 0.8);
	\teleportUU{0, 0.4}{0.6, 0.4};
	\end{tikzpicture}
	\ -\
	\begin{tikzpicture}[anchorbase]
	\draw[-] (0, 0) -- (0, 0.3) arc(180:0:0.3) -- (0.6, 0);
	\draw[-] (0, 1.5) -- (0, 1.2) arc(180:360:0.3) -- (0.6, 1.5);
	\teleportRR{0.3, 0.6}{0.3, 0.9};
	\end{tikzpicture}, \quad \quad \quad
	\begin{tikzpicture}[anchorbase]
	\draw[-] (0, 0.5) -- (0, 0.25) arc(180:360:0.25) -- (0.5, 0.5);
	\adot{(0, 0.35)};
	\end{tikzpicture}\
	\ =\ 
	-\ \begin{tikzpicture}[anchorbase]
	\draw[-] (0, 0.5) -- (0, 0.25) arc(180:360:0.25) -- (0.5, 0.5);
	\adot{(0.5, 0.35)};
	\end{tikzpicture}.
	\end{equation}
\end{lem}

\begin{proof}
	Adjoining a crossing to the top and bottom of \eqref{rel:dotCrossingSlide} yields
	\begin{equation*}
	\begin{tikzpicture}[anchorbase]
	\draw [-] (0, -0.2) -- (0,0) to[out=90,in=270] (0.6,0.6) to[out=90,in=270] (0, 1.2) to[out=90,in=270] (0.6, 1.8) -- (0.6, 2);
	\draw [-] (0.6, -0.2) -- (0.6,0) to[out=90,in=270] (0,0.6) to[out=90,in=270] (0.6,1.2) to[out=90,in=270] (0,1.8) -- (0, 2);
	\adot{(0.05,1.35)};
	\end{tikzpicture} - 
	\begin{tikzpicture}[anchorbase]
	\draw [-] (0, -0.2) -- (0,0) to[out=90,in=270] (0.6,0.6) to[out=90,in=270] (0, 1.2) to[out=90,in=270] (0.6, 1.8) -- (0.6, 2);
	\draw [-] (0.6, -0.2) -- (0.6,0) to[out=90,in=270] (0,0.6) to[out=90,in=270] (0.6,1.2) to[out=90,in=270] (0,1.8) -- (0, 2);
	\adot{(0.55,0.75)};
	\end{tikzpicture} \ = \begin{tikzpicture}[anchorbase]
	\draw[-] (0.5, -0.8) -- (0.5, -0.6) to[out=90,in=270] (0, 0) -- (0, 0.6) to[out=90,in=270] (0.5, 1.2) -- (0.5, 1.4);
	\draw[-] (0, -0.8) -- (0, -0.6) to[out=90,in=270] (0.5, 0) -- (0.5, 0.6) to[out=90, in=270] (0, 1.2) -- (0, 1.4);
	\teleportUU{0, 0.3}{0.5, 0.3};
	\end{tikzpicture} - \begin{tikzpicture}[anchorbase]
	\draw[-] (0.5, -0.325) -- (0.5, -0.2) to[out=90,in=270] (0, 0.4) arc(180:0:0.25) to[out=270,in=90] (0, -0.2) -- (0, -0.325);
	\draw[-] (0.5, 1.875) -- (0.5, 1.75) to[out=270,in=90] (0, 1.15) arc(180:360:0.25) to[out=90, in=270] (0, 1.75) -- (0, 1.875);
	\teleportRR{0.25, 0.65}{0.25, 0.9};
	\end{tikzpicture}\ .
	\end{equation*}
	The first relation from \eqref{rel:mirroredDotSlide} follows after sliding the teleporter endpoints through crossings, using the first relation from \eqref{rel:brauer} several times, the fourth relation from \eqref{rel:brauer} and the first relation from \eqref{rel:brauerBasics} on the last diagram, and reversing the orientation of both teleporter endpoints in the last diagram. For the second relation, we compute:
	\begin{equation*}
		\begin{tikzpicture}[anchorbase]
		\draw[-] (0, 0.5) -- (0, 0.25) arc(180:360:0.25) -- (0.5, 0.5);
		\adot{(0, 0.35)};
		\end{tikzpicture}
		\quad \overset{\eqref{rel:brauer}}{=} \quad 
		\begin{tikzpicture}[anchorbase]
		\draw[-] (-1, 1.5) -- (-1, 0.75) arc(180:360:0.25) -- (-0.5, 0.75) arc(180:0:0.25) -- (0, 0.75) -- (0, 0.25) arc(180:360:0.25) -- (0.5, 1.25);
		\adot{(0, 0.35)};
		\end{tikzpicture}
		\quad \underset{\eqref{eq:superInterchangeLaw}}{\overset{\eqref{rel:dotCapSlide}}{=}} \quad  
		-\ \begin{tikzpicture}[anchorbase]
		\draw[-] (-1, 1.25) -- (-1, 0.5) arc(180:360:0.25) -- (-0.5, 0.75) arc(180:0:0.25) -- (0, 0.75) -- (0, 0.25) arc(180:360:0.25) -- (0.5, 1.25);
		\adot{(-0.5, 0.6)};
		\end{tikzpicture}
		\quad
		\overset{\eqref{eq:superInterchangeLaw}}{=} \quad
		-\ \begin{tikzpicture}[anchorbase, xscale=-1]
		\draw[-] (-1, 1.5) -- (-1, 0.75) arc(180:360:0.25) -- (-0.5, 0.75) arc(180:0:0.25) -- (0, 0.75) -- (0, 0.25) arc(180:360:0.25) -- (0.5, 1.25);
		\adot{(0, 0.35)};
		\end{tikzpicture}
		\quad \overset{\eqref{rel:brauer}}{=} \quad
		-\ \begin{tikzpicture}[anchorbase]
		\draw[-] (0, 0.5) -- (0, 0.25) arc(180:360:0.25) -- (0.5, 0.5);
		\adot{(0.5, 0.35)};
		\end{tikzpicture}. \qedhere
	\end{equation*}
\end{proof}

Recall that $\varphi$ denotes a unimodular $(\nu, -^\star)$-superhermitian form on $V = A^{m \mid n}$, $\Phi = \tr \circ\,\varphi$ is the corresponding nondegenerate $(\nu, -^\star)$-supersymmetric form, and $\g = \osp(\varphi)$. In the following, left duals for the basis $B_\g$ of $\g$ are taken with respect to the nondegenerate bilinear form given by $(X, Y) \mapsto \tr(\str_\varphi(XY))$, as discussed in Lemma \ref{lem:strNondegenerate2}.

\begin{defin} \label{def:omega}
	We define the \emph{quadratic Casimir elements} $\Omega = \summ_{X \in B_\g} X \otimes X^\vee \in \g \otimes \g$ and $C = \summ_{X \in B_\g} XX^\vee \in U(\g)$, where $U(\g)$ denotes the universal enveloping algebra of $\g$.
\end{defin}

Note that $\Omega$ and $C$ are even. A quick calculation shows that they are both independent of the choice of basis $B_\g$.

The proof of the following theorem will occupy the remainder of this section.

\begin{theo} \label{thm:affineFunctor}
	There exists a monoidal superfunctor $\hat{F}_\Phi \colon \AB(A, -^\star) \to \End(\g\dashsmod)$ such that $\hat{F}_\Phi(\go) = V \otimes -$,
	\begin{align*}&\hat{F}_\Phi(f) = F_\Phi(f) \otimes - \quad \text{ for all } f \in \{\begin{tikzpicture}[anchorbase]
	\draw[-] (0, 0) -- (0.4, 0.4);
	\draw[-] (0.4, 0) -- (0, 0.4);
	\end{tikzpicture}, \begin{tikzpicture}[anchorbase]
	\draw[-] (0, 0) -- (0, 0.2) arc(180:0:0.2) -- (0.4, 0);
	\end{tikzpicture}, \begin{tikzpicture}[anchorbase, yscale=-1]
	\draw[-] (0, 0) -- (0, 0.2) arc(180:0:0.2) -- (0.4, 0);
	\end{tikzpicture}, \begin{tikzpicture}[anchorbase]
	\draw[-] (0, 0) -- (0, 0.4);
	\token{0, 0.2}{west}{a};
	\end{tikzpicture} : a \in A\},\\
	&\hat{F}_\Phi\left(\begin{tikzpicture}[anchorbase]
	\draw[-] (0, 0) -- (0, 0.4);
	\adot{(0, 0.2)};
	\end{tikzpicture}\right) = \nu\left(\Delta(C) - 1 \otimes C\right).
	\end{align*}
	Here, $F_\Phi$ is the incarnation functor of Proposition \ref{prop:incarnation}, $\Delta \colon U(\g) \to U(\g) \otimes U(\g)$ is the usual comultiplication map, and $\Delta(C) - 1 \otimes C$ denotes the supernatural transformation from $V \otimes -$ to $V \otimes -$ with components given by $$v \otimes w \mapsto (\Delta(C) - 1 \otimes C)(v \otimes w).$$
\end{theo} 
Throughout the rest of the paper, we will identify elements of tensor powers of $U(\g)$ and their associated supernatural transformations in the same way we did with $\Delta(C) - 1 \otimes C$ above.

A calculation analogous to the proof of Proposition \ref{prop:teleporters} shows that the action of $C$ intertwines the action of $\g$ on $V$. It quickly follows that the image of the affine dot is indeed an even supernatural transformation. Another direct computation shows that 
	\begin{equation}
		\Delta(C) - 1 \otimes C = C \otimes 1 + 2\Omega.
	\end{equation}

\begin{rem} \label{rem:oddCase}
	In \cite[Thm.~10.1]{diagrammaticsRealSupergroups}, the incarnation functor $F_\Phi$ of Proposition \ref{prop:incarnation} was shown to exist even when the nondegenerate supersymmetric form $\Phi$ is odd. In the affine case, it is not straightforward to generalize to odd forms. Briefly, supposing that $\bar{\Phi} = 1$, the duals $X^\vee$ appearing in Definition \ref{def:omega} satisfy $\overline{X^\vee} = \bar{X} + 1$, and hence $C$ is odd. As the action of $C$ supercommutes with the action of $U(\g)$, we get that $C^2$ acts as 0. Schur's Lemma then tells us that $C$ acts on each irreducible supermodule as 0. Thus we restrict our attention to even forms in this paper.
\end{rem}

We now proceed towards proving Theorem \ref{thm:affineFunctor}. Consider the functor $R \colon \g\dashsmod \to \End(\g\dashsmod)$ given on objects by $X \mapsto X \otimes -$ and on morphisms by $f \mapsto f \otimes -$. The composite map ${R \circ F_{\Phi} \colon \B(A, -^\star) \to \End(\g\dashsmod)}$ is a functor whose action on the generating morphisms $\begin{tikzpicture}[anchorbase]
\draw[-] (0, 0) -- (0.4, 0.4);
\draw[-] (0.4, 0) -- (0, 0.4);
\end{tikzpicture}\,{,}\ \begin{tikzpicture}[anchorbase]
\draw[-] (0, 0) -- (0, 0.2) arc(180:0:0.2) -- (0.4, 0);
\end{tikzpicture}\,{,}$ $\begin{tikzpicture}[anchorbase, yscale=-1]
\draw[-] (0, 0) -- (0, 0.2) arc(180:0:0.2) -- (0.4, 0);
\end{tikzpicture}\,{,}$ and $\begin{tikzpicture}[anchorbase]
\draw[-] (0, 0) -- (0, 0.4);
\token{0, 0.2}{west}{a};
\end{tikzpicture}$ coincides with that of $\hat{F_\Phi}$. This implies that $\hat{F_\Phi}$ preserves all of the defining relations for $\B(A, -^\star)$, i.e.\ the defining relations for $\AB(A, -^\star)$ that do not involve the affine dot. As such, all that remains is to show that the relations involving the affine dot, namely \eqref{rel:dotCrossingSlide}, \eqref{rel:dotCapSlide}, and \eqref{rel:dotTokenCommuting}, are also preserved by $\hat{F_\Phi}$.

\begin{defin} \label{def:subOmegas}
	Define $$\Omega^{12} = \summ_{X \in B_\g} X \otimes X^\vee \otimes 1, \quad \Omega^{13} = \summ_{X \in B_\g}X \otimes 1 \otimes X^\vee, \quad \Omega^{23} = \summ_{X \in B_\g} 1 \otimes X \otimes X^\vee.$$
\end{defin}

\begin{lem}
	We have 
	\begin{equation} \label{eq:leftDot}
	\hat{F}_\Phi\left(\begin{tikzpicture}[anchorbase]
	\draw[-] (0, 0) -- (0, 0.6);
	\draw[-] (0.4, 0) -- (0.4, 0.6);
	\adot{(0, 0.3)};
	\end{tikzpicture} \right) =\nu\left( C \otimes 1 \otimes 1 + 2( \Omega^{12} + \Omega^{13})\right),
	\end{equation}
	\begin{equation} \label{eq:rightDot}
	\hat{F}_\Phi\left(\begin{tikzpicture}[anchorbase]
	\draw[-] (0, 0) -- (0, 0.6);
	\draw[-] (0.4, 0) -- (0.4, 0.6);
	\adot{(0.4, 0.3)};
	\end{tikzpicture} \right) = \nu\left(1 \otimes C \otimes 1 + 2\Omega^{23}\right).
	\end{equation}
\end{lem}

\begin{proof}
	Let $W$ be a left $\g$-supermodule, and $u, v \in V, w \in W$. We have that $\hat{F}_\Phi\left(\begin{tikzpicture}[anchorbase]
	\draw[-] (0, 0) -- (0, 0.6);
	\draw[-] (0.4, 0) -- (0.4, 0.6);
	\adot{(0, 0.3)};
	\end{tikzpicture} \right)_W$ acts as $u \otimes v \otimes w \mapsto \nu(C \otimes 1 + 2\Omega)(u \otimes (v \otimes w))$, and
	\begin{align*}
	&\hspace{-55px} (C \otimes 1 + 2\Omega)(u \otimes (v \otimes w))\\
	&= Cu \otimes v \otimes w + 2 \summ_{X \in B_\g} (-1)^{\bar{X}\bar{u}} Xu \otimes X^\vee(v \otimes w)\\
	&= Cu \otimes v \otimes w + 2\summ_{X \in B_\g} (-1)^{\bar{X}\bar{u}}Xu \otimes X^\vee v \otimes w + (-1)^{\bar{X} \bar{u} + \bar{X} \bar{v}}Xu \otimes v \otimes X^\vee w\\
	&= \left(C \otimes 1 \otimes 1 + 2\summ_{X \in B_\g} X \otimes X^\vee \otimes 1 + 2\summ_{X \in B_\g} X \otimes 1 \otimes X^\vee\right)(u \otimes v \otimes w)\\
	&= (C \otimes 1 \otimes 1 + 2(\Omega^{12} +  \Omega^{13}))(u \otimes v \otimes w).
	\end{align*}
	For the second claim, we have that $\hat{F}_\Phi\left(\begin{tikzpicture}[anchorbase]
	\draw[-] (0, 0) -- (0, 0.6);
	\draw[-] (0.4, 0) -- (0.4, 0.6);
	\adot{(0.4, 0.3)};
	\end{tikzpicture} \right)_W$ is equal to $$\id_{V \otimes V \otimes W} \circ \nu(\id_V \otimes (C \otimes 1 + 2\Omega)) = \nu\left(1 \otimes C \otimes 1 + 2\Omega^{23}\right),$$
	as desired.
\end{proof}

\begin{lem}
	For all $u, v \in V$, we have 
	\begin{equation} \label{eq:transferC}
	\Phi(Cu, v) = \Phi(u, Cv).
	\end{equation}
\end{lem}

\begin{proof}
	\begin{multline*}
	\Phi(Cu, v)
	=
	\summ_{X \in B_\g} \Phi(XX^\vee u, v)
	{\overset{\eqref{def:dagger}}{=}}
	\summ_{X \in B_\g}(-1)^{\bar{X}}\Phi(u, (X^\vee)^\dagger X^\dagger v)
	\overset{\eqref{def:osp}}{=}
	\summ_{X \in B_\g} (-1)^{\bar{X}} \Phi(u, X^\vee X v)
	\\
	\overset{\eqref{eq:doubleDual}}{=}
	\summ_{X \in B_\g} \Phi(u, X^\vee (X^\vee)^\vee v)
	=
	\summ_{X \in B_\g} \Phi(u, XX^\vee v)
	=
	\Phi(u, Cv),
	\end{multline*}
	where we switch to a sum over the dual basis in the second-to-last equality.
\end{proof}

\begin{prop} \label{prop:dotCapSlide}
	We have $\hat{F}_\Phi\left(\begin{tikzpicture}[anchorbase]
	\draw[-] (0, 0) -- (0, 0.3) arc(180:0:0.3) -- (0.6, 0);
	\adot{(0, 0.2)}; 
	\end{tikzpicture}\right)
	\ =\
	- \hat{F}_\Phi\left(
	\begin{tikzpicture}[anchorbase]
	\draw[-] (0, 0) -- (0, 0.3) arc(180:0:0.3) -- (0.6, 0);
	\adot{(0.6, 0.2)}; 
	\end{tikzpicture}\right)$.
\end{prop}

\begin{proof}
	Let $W$ be a left $\g$-supermodule and $u, v \in V, w \in W$. We compute:
	\begin{align*}
	&\hspace{-40px} \hat{F}_\Phi\left(\begin{tikzpicture}[anchorbase]
	\draw[-] (0, 0) -- (0, 0.3) arc(180:0:0.3) -- (0.6, 0);
	\adot{(0, 0.2)}; 
	\end{tikzpicture}\right)_W(u \otimes v \otimes w)
	\\
	&\overset{\mathclap{\eqref{eq:leftDot}}}{=}\ \,\nu\left( F_\Phi\left(\begin{tikzpicture}[anchorbase]
	\draw[-] (0, 0) -- (0, 0.3) arc(180:0:0.3) -- (0.6, 0);
	\end{tikzpicture}\right) \otimes \id_W \right) \left(C \otimes 1 \otimes 1 + 2(\Omega^{12} + \Omega^{13}) \right)(u \otimes v \otimes w)\\
	&= \nu\Phi(Cu, v)w + 2\nu\left(\summ_{X \in B_\g} (-1)^{\bar{X}\bar{u}}\Phi(Xu, X^\vee v)w + \summ_{X \in B_\g}(-1)^{\bar{X}(\bar{u} + \bar{v})}\Phi(Xu, v)X^\vee w\right)\\
	&\underset{\mathclap{\eqref{def:osp}}}{\overset{\mathclap{\eqref{def:dagger}}}{=}}\ \,\nu\Phi(Cu, v)w - 2\nu\left(\summ_{X \in B_\g} \Phi(u, XX^\vee v)w + \summ_{X \in B_\g}(-1)^{\bar{X}\bar{v}}\Phi(u, Xv)X^\vee w\right)\\
	&\overset{\mathclap{\eqref{eq:transferC}}}{=}\ \,\nu\Phi(u, Cv)w - 2\nu\left(\Phi(u, Cv)w + \summ_{X \in B_\g} (-1)^{\bar{X}\bar{v}}\Phi(u, Xv)X^\vee w\right)\\
	&= -\nu\Phi(u, Cv)w - 2\nu\summ_{X\in B_\g} (-1)^{\bar{X}\bar{v}}\Phi(u, Xv)X^\vee w\\
	&= -\nu\left( F_\Phi\left(\begin{tikzpicture}[anchorbase]
	\draw[-] (0, 0) -- (0, 0.3) arc(180:0:0.3) -- (0.6, 0);
	\end{tikzpicture}\right) \otimes \id_W \right)(1 \otimes C \otimes 1 + 2\Omega^{23})(u \otimes v \otimes w)\\
	&\overset{\mathclap{\eqref{eq:rightDot}}}{=}\ \,- \hat{F}_\Phi\left(\begin{tikzpicture}[anchorbase]
	\draw[-] (0, 0) -- (0, 0.3) arc(180:0:0.3) -- (0.6, 0);
	\adot{(0.6, 0.2)}; 
	\end{tikzpicture}\right)_W(u \otimes v \otimes w),
	\end{align*}
	as desired. 
\end{proof}

\begin{prop} \label{prop:dotTokenPreserved}
	For all $a \in A$, we have $\hat{F}_\Phi\left(\begin{tikzpicture}[anchorbase]
	\draw[-] (0, 0) -- (0, 0.8);
	\adot{(0, 0.25)};
	\token{0, 0.55}{west}{a};
	\end{tikzpicture}
	\right)
	\ =\
	\hat{F}_\Phi\left(
	\begin{tikzpicture}[anchorbase]
	\draw[-] (0, 0) -- (0, 0.8);
	\adot{(0, 0.55)};
	\token{0, 0.25}{west}{a};
	\end{tikzpicture}\right)$.
\end{prop}

\begin{proof}
	Let $W$ be a left $\g$-supermodule, $a \in A$, $v \in V$, and $w \in W$. We compute:
	\begin{multline*}
	\hat{F}_\Phi\left(\begin{tikzpicture}[anchorbase]
	\draw[-] (0, 0) -- (0, 0.8);
	\adot{(0, 0.25)};
	\token{0, 0.55}{west}{a};
	\end{tikzpicture}
	\right)_W(v \otimes w)
	=
	\nu\hat{F}_\Phi\left(\begin{tikzpicture}[anchorbase]
	\draw[-] (0, 0) -- (0, 0.8);
	\token{0, 0.4}{west}{a};
	\end{tikzpicture}\right)_W\left(Cv \otimes w + 2\summ_{X \in B_\g} (-1)^{\bar{X}\bar{v}}Xv \otimes X^\vee w\right)\\
	=
	\nu(-1)^{\bar{a}\bar{v}}Cva^\star \otimes w + 2\nu\summ_{X \in B_\g} (-1)^{\bar{X}\bar{v} + \bar{a}\bar{X} + \bar{a}\bar{v}} Xv a^\star \otimes X^\vee w\\
	=
	\nu(C \otimes 1 + 2\Omega)((-1)^{\bar{a}\bar{v}}va^\star \otimes w)
	=
	\hat{F}_\Phi\left(\begin{tikzpicture}[anchorbase]
	\draw[-] (0, 0) -- (0, 0.8);
	\adot{(0, 0.4)};
	\end{tikzpicture}\right)_W\left((-1)^{\bar{a}\bar{v}}va^\star \otimes w \right)
	= \hat{F}_\Phi\left(
	\begin{tikzpicture}[anchorbase]
	\draw[-] (0, 0) -- (0, 0.8);
	\adot{(0, 0.25)};
	\token{0, 0.55}{west}{a};
	\end{tikzpicture}\right)_W (v \otimes w). \qedhere
	\end{multline*}
\end{proof}

Fix a homogeneous $\kk$-basis $B_{\End}$ for $\End_A(V)$. As for $B_{\g}$, we take duals for this basis relative to the nondegenerate bilinear form $(X, Y) \mapsto \tr(\str_\varphi(XY))$; see Lemma \ref{lem:strNondegenerate1}.

\begin{lem} \label{lem:omegaDecomposition}
	We have 
	\begin{equation} \label{eq:omegaDecomposition}
		2\Omega = \summ_{X \in B_{\End}}X \otimes X^\vee - \summ_{X \in B_{\End}} X^\dagger \otimes X^\vee.
	\end{equation}
\end{lem}

\begin{proof}
	A quick calculation shows that the sums appearing in the lemma statement are independent of the basis for $\End_A(V)$. Since $-^\dagger$ is an involutive linear operator on $\End_A(V)$, $\End_A(V)$ decomposes a direct sum of its 1-eigenspace and its $(-1)$-eigenspace. Let $B_1$ be a basis for the 1-eigenspace. Noting that $\g$ is precisely the $(-1)$-eigenspace, we find that $B_1 \dot\cup B_\g$ is a basis for $\End_A(V)$. Hence we have:
	\begin{align*}
	2\Omega &= 2\summ_{X \in B_\g} X \otimes X^\vee\\
	&= \summ_{X \in B_\g}  X \otimes X^\vee + \summ_{X \in B_1} X \otimes X^\vee + \summ_{X \in B_\g} X \otimes X^\vee - \summ_{X \in B_1} X \otimes X^\vee\\
	&= \summ_{X \in B_\g} X \otimes X^\vee + \summ_{X \in B_1} X \otimes X^\vee - \summ_{X \in B_\g} X^\dagger \otimes X^\vee - \summ_{X \in B_1} X^\dagger \otimes X^\vee\\
	&=\summ_{X \in B_{\End}} X \otimes X^\vee - \summ_{X \in B_{\End}} X^\dagger \otimes X^\vee. \qedhere
	\end{align*}
\end{proof}

\begin{lem} \label{lem:omegaSwap}
	For all $u, v \in V$ we have:
	\begin{equation}\label{eq:omegaSwap}
	\left(\summ_{X \in B_{\End}} X \otimes X^\vee \right)(u \otimes v) = (-1)^{\bar{u}\bar{v}}v \otimes u \left(\summ_{b \in B_A}b \otimes b^\vee\right),
	\end{equation}
	\begin{equation} \label{eq:omegaDagger}
	\left(\summ_{X \in B_{\End}} X^\dagger \otimes X^\vee \right)(u \otimes v) = \nu\summ_{b \in B_A} \summ_{w \in B_V} \Phi(u, vb)w \otimes w^\vee b^\vee.
	\end{equation}
\end{lem}

\begin{proof}
	The first identity is analogous to that of \cite[Lem.~3.7]{mcsween2021affine}, and follows from essentially the same proof; the only real distinction is that we work with right $A$-supermodules in the present paper, whereas the authors of \cite{mcsween2021affine} work with left $A$-supermodules. (Note, however, that $\Omega$ in \cite{mcsween2021affine} refers to the element $\summ_{X \in B_{\End}} X \otimes X^\vee$ appearing in \eqref{eq:omegaSwap}, not our $\Omega$ from Definition \ref{def:omega}.) For the second identity, we compute:
	\begin{align*}
	\left(\summ_{X \in B_{\End}} X^\dagger \otimes X^\vee \right)(u \otimes v) &= \summ_{X \in B_{\End}} (-1)^{\bar{X}\bar{u}} X^\dagger u \otimes X^\vee v\\
	&\overset{\mathclap{\eqref{eq:traceDecomposition}}}{=}\ \,\summ_{X \in B_{\End}}\summ_{w \in B_V} (-1)^{\bar{X}\bar{u}}\Phi(w^\vee, X^\dagger u)w \otimes X^\vee v\\
	&\overset{\mathclap{\eqref{def:dagger}}}{=}\ \,\summ_{X \in B_{\End}}\summ_{w \in B_V} (-1)^{\bar{X}\bar{u} + \bar{X}\bar{w}}\Phi(Xw^\vee,  u)w \otimes X^\vee v\\
	&\overset{\mathclap{\eqref{def:superSymmetric}}}{=}\ \,\nu \summ_{X \in B_{\End}}\summ_{w \in B_V} (-1)^{\bar{X}\bar{w} + \bar{u}\bar{w}}\Phi(u, X w^\vee) w \otimes X^\vee v\\
	&\overset{\mathclap{\eqref{eq:omegaSwap}}}{=}\ \,\nu\summ_{b \in B_A} \summ_{w \in B_V} (-1)^{\bar{u}\bar{w} + \bar{w}\bar{v} + \bar{b}\bar{w}}\Phi(u, vb)w \otimes w^\vee b^\vee\\
	&= \nu\summ_{b \in B_A} \summ_{w \in B_V}\Phi(u, vb)w \otimes w^\vee b^\vee,
	\end{align*}
	as desired. Note that in the last equality we simplified the sign by using the fact that the only nonzero summands are those for which $\bar{u} = \bar{v} + \bar{b}$.
\end{proof}

\begin{prop} \label{prop:dotSlidePreserved}
	We have $\hat{F}_\Phi\left(\begin{tikzpicture}[anchorbase]
	\draw[-] (0, 0) -- (0.8, 0.8);
	\draw[-] (0.8, 0) -- (0, 0.8);
	\adot{(0.2, 0.6)};
	\end{tikzpicture}
	\ -\
	\begin{tikzpicture}[anchorbase]
	\draw[-] (0, 0) -- (0.8, 0.8);
	\draw[-] (0.8, 0) -- (0, 0.8);
	\adot{(0.6, 0.2)};
	\end{tikzpicture}\right)
	\ =\
	\hat{F}_\Phi\left(
	\begin{tikzpicture}[anchorbase]
	\draw[-] (0, 0) -- (0, 0.8);
	\draw[-] (0.6, 0) -- (0.6, 0.8);
	\teleportUU{0, 0.4}{0.6, 0.4};
	\end{tikzpicture}
	\ -\
	\begin{tikzpicture}[anchorbase]
	\draw[-] (0, 0) -- (0, 0.3) arc(180:0:0.3) -- (0.6, 0);
	\draw[-] (0, 1.5) -- (0, 1.2) arc(180:360:0.3) -- (0.6, 1.5);
	\teleportRR{0.3, 0.6}{0.3, 0.9};
	\end{tikzpicture}\right)$.
\end{prop}

\begin{proof}
	Given that we know $\hat{F}_\Phi$ respects the first relation from \eqref{rel:brauer}, it suffices to prove the equality obtained by adjoining a crossing to the top of each of the four diagrams. Using the aforementioned relation, the first relation from \eqref{rel:brauerBasics}, and \eqref{rel:teleporterCrossingSlide}, our goal reduces to showing that $$\hat{F}_\Phi\left(	\begin{tikzpicture}[anchorbase]
	\draw[-] (0, 0) -- (0, 0.2) to[out=90, in=270] (0.4, 0.6) to[out=90, in=270] (0, 1) -- (0, 1.2);
	\draw[-] (0.4, 0) -- (0.4, 0.2) to[out=90, in=270] (0, 0.6) to[out=90, in=270] (0.4, 1) -- (0.4, 1.2);
	\adot{(0, 0.6)};
	\end{tikzpicture}
	\ -\
	\begin{tikzpicture}[anchorbase]
	\draw[-] (0, 0) -- (0, 1.2);
	\draw[-] (0.4, 0) -- (0.4, 1.2);
	\adot{(0.4, 0.6)};
	\end{tikzpicture}\right)
	\ =\
	\hat{F}_\Phi\left(
	\begin{tikzpicture}[anchorbase]
	\draw[-] (0, 0) -- (0.8, 0.8);
	\draw[-] (0.8, 0) -- (0, 0.8);
	\teleportUU{0.2, 0.6}{0.6, 0.6};
	\end{tikzpicture}
	\ -\
	\begin{tikzpicture}[anchorbase]
	\draw[-] (0, 0) -- (0, 0.3) arc(180:0:0.3) -- (0.6, 0);
	\draw[-] (0, 1.5) -- (0, 1.2) arc(180:360:0.3) -- (0.6, 1.5);
	\teleportUU{0.6, 0.2}{0.6, 1.3};
	\end{tikzpicture}\right).$$
	Let $W$ be a left $\g$-supermodule, $u, v \in V$, and $w \in W$. We compute:
	\begin{align*}
	&\hspace{-40px} \hat{F}_\Phi\left(\begin{tikzpicture}[anchorbase]
	\draw[-] (0, 0) -- (0, 0.2) to[out=90, in=270] (0.4, 0.6) to[out=90, in=270] (0, 1) -- (0, 1.2);
	\draw[-] (0.4, 0) -- (0.4, 0.2) to[out=90, in=270] (0, 0.6) to[out=90, in=270] (0.4, 1) -- (0.4, 1.2);
	\adot{(0, 0.6)};
	\end{tikzpicture}\right)_W(u \otimes v \otimes w) \\
	&= \nu (-1)^{\bar{u}\bar{v}}\hat{F}_\Phi\left(\begin{tikzpicture}[anchorbase]
	\draw[-] (0, 0) -- (0.8, 0.8);
	\draw[-] (0.8, 0) -- (0, 0.8);
	\adot{(0.2, 0.2)};
	\end{tikzpicture}\right)_W(v \otimes u \otimes w)\\
	&\overset{\mathclap{\eqref{eq:leftDot}}}{=}\ \, (-1)^{\bar{u}\bar{v}}\hat{F}_\Phi\left(\begin{tikzpicture}[anchorbase]
	\draw[-] (0, 0) -- (0.8, 0.8);
	\draw[-] (0.8, 0) -- (0, 0.8);
	\end{tikzpicture}\right)_W\bigg(Cv \otimes u \otimes w + 2\summ_{X \in B_\g} (-1)^{\bar{X}\bar{v}} Xv \otimes X^\vee u \otimes w\\
	&\hphantom{=================} + 2\summ_{X \in B_\g} (-1)^{\bar{X}\bar{v} + \bar{X}\bar{u}} Xv \otimes u \otimes X^\vee w\bigg)\\
	&= \nu(-1)^{\bar{u}\bar{v}}\bigg((-1)^{\bar{v}\bar{u}}u \otimes Cv \otimes w + 2\summ_{X \in B_\g} (-1)^{\bar{X}\bar{v} + (\bar{X} + \bar{v})(\bar{X} + \bar{u})} X^\vee u \otimes X v \otimes w\\
	&\hphantom{=======.} + 2\summ_{X \in B_\g} (-1)^{\bar{X}\bar{v} + \bar{u}\bar{v}} u \otimes Xv \otimes X^\vee w\bigg)\\
	&= \nu\left(u \otimes Cv \otimes w + 2\summ_{X \in B_\g} (-1)^{\bar{X} + \bar{X}\bar{u}} X^\vee u \otimes Xv \otimes w + 2\summ_{X \in B_\g} (-1)^{\bar{X}\bar{v}} u \otimes Xv \otimes X^\vee w\right)\\
	&\overset{\mathclap{\eqref{eq:doubleDual}}}{=}\ \,
	\nu\left(u \otimes Cv \otimes w + 2\summ_{X \in B_\g} (-1)^{\bar{X}\bar{u}}Xu \otimes X^\vee v \otimes w + 2\summ_{X \in B_{\g}} (-1)^{\bar{X}\bar{v}}u \otimes Xv \otimes X^\vee w\right)\\
	&= \nu(1 \otimes C \otimes 1 + 2\Omega^{12} + 2\Omega^{23})(u \otimes v \otimes w).
	\end{align*}
	Hence, using \eqref{eq:rightDot}, we find that $\hat{F}_\Phi\left(	\begin{tikzpicture}[anchorbase]
	\draw[-] (0, 0) -- (0, 0.2) to[out=90, in=270] (0.4, 0.6) to[out=90, in=270] (0, 1) -- (0, 1.2);
	\draw[-] (0.4, 0) -- (0.4, 0.2) to[out=90, in=270] (0, 0.6) to[out=90, in=270] (0.4, 1) -- (0.4, 1.2);
	\adot{(0, 0.6)};
	\end{tikzpicture}
	\ -\
	\begin{tikzpicture}[anchorbase]
	\draw[-] (0, 0) -- (0, 1.2);
	\draw[-] (0.4, 0) -- (0.4, 1.2);
	\adot{(0.4, 0.6)};
	\end{tikzpicture}\right) = 2 \nu\Omega^{12}$. Next, we have
	\begin{align*}
	\hat{F}_\Phi\left(\begin{tikzpicture}[anchorbase]
	\draw[-] (0, 0) -- (0, 0.3) arc(180:0:0.3) -- (0.6, 0);
	\draw[-] (0, 1.5) -- (0, 1.2) arc(180:360:0.3) -- (0.6, 1.5);
	\teleportUU{0.6, 0.2}{0.6, 1.3};
	\end{tikzpicture}\right)_W(u \otimes v \otimes w) &= \summ_{b \in B_A}\summ_{x \in B_V} (-1)^{\bar{b}\bar{u} + \bar{b}\bar{v}}\Phi(u, vb^\vee)x \otimes x^\vee b \otimes w\\
	&\overset{\mathclap{\eqref{eq:doubleDual}}}{=}\ \,\summ_{b \in B_A}\summ_{x \in B_V} \Phi(u, vb)x \otimes x^\vee b^\vee \otimes w,
	\end{align*}
	where in the last equality we switch to the dual basis of $B_A$ and use the fact that the summands are zero unless $\bar{u} + \bar{v} + \bar{b} = 0$ to simplify the sign. With this in mind, we compute:
	\begin{align*}
	&\hspace{-45px} \hat{F}_\Phi\left(	\begin{tikzpicture}[anchorbase]
	\draw[-] (0, 0) -- (0, 0.2) to[out=90, in=270] (0.4, 0.6) to[out=90, in=270] (0, 1) -- (0, 1.2);
	\draw[-] (0.4, 0) -- (0.4, 0.2) to[out=90, in=270] (0, 0.6) to[out=90, in=270] (0.4, 1) -- (0.4, 1.2);
	\adot{(0, 0.6)};
	\end{tikzpicture}
	\ -\
	\begin{tikzpicture}[anchorbase]
	\draw[-] (0, 0) -- (0, 1.2);
	\draw[-] (0.4, 0) -- (0.4, 1.2);
	\adot{(0.4, 0.6)};
	\end{tikzpicture}\right)_W(u \otimes v \otimes w)\\
	&= 2\nu\Omega^{12}(u \otimes v \otimes w)\\
	& \overset{\mathclap{\eqref{eq:omegaDecomposition}}}{=}\ \, \nu\left(\summ_{X \in B_{\End}}(X \otimes X^\vee - X^\dagger \otimes X^\vee) u \otimes v \right) \otimes w\\
	&\underset{\mathclap{\eqref{eq:omegaDagger}}}{\overset{\mathclap{\eqref{eq:omegaSwap}}}{=}}\ \, \nu\left((-1)^{\bar{u}\bar{v}}v \otimes u \summ_{b \in B_A} b \otimes b^\vee - \nu \summ_{b \in B_A} \summ_{x \in B_V} \Phi(u, vb) x \otimes x^\vee b^\vee \right) \otimes w\\
	&= \hat{F}_\Phi\left(
	\begin{tikzpicture}[anchorbase]
	\draw[-] (0, 0) -- (0.8, 0.8);
	\draw[-] (0.8, 0) -- (0, 0.8);
	\teleportUU{0.2, 0.6}{0.6, 0.6};
	\end{tikzpicture}\right)_W(u \otimes v \otimes w) - \hat{F}_\Phi\left(\begin{tikzpicture}[anchorbase]
	\draw[-] (0, 0) -- (0, 0.3) arc(180:0:0.3) -- (0.6, 0);
	\draw[-] (0, 1.5) -- (0, 1.2) arc(180:360:0.3) -- (0.6, 1.5);
	\teleportUU{0.6, 0.2}{0.6, 1.3};
	\end{tikzpicture}\right)_W(u \otimes v \otimes w),
	\end{align*}
	proving the desired equality.
\end{proof}

Recall that Proposition \ref{prop:incarnation} implies that $\hat{F}_\Phi$ preserves all of the defining relations for $\AB(A, -^\star)$ that do not involve affine dots. Propositions \ref{prop:dotCapSlide}, \ref{prop:dotTokenPreserved}, and \ref{prop:dotSlidePreserved} show that the relations involving affine dots are also preserved, completing the proof of Theorem \ref{thm:affineFunctor}.

\section{Basis Conjecture} \label{s:basis}

In this final section, we state a conjecture for bases of homomorphism spaces in $\AB(A, -^\star)$ and discuss a potential proof strategy. We begin by proving a pair of results that will inform the basis conjecture. For each $n \in \N$, we define the shorthand $\begin{tikzpicture}[anchorbase]
\draw[-] (0, 0) -- (0, 0.4);
\adot{(0, 0.2)} node[anchor=east, color=black] {$\scr n$};
\end{tikzpicture} := \left(\begin{tikzpicture}[anchorbase]
\draw[-] (0, 0) -- (0, 0.4);
\adot{(0, 0.2)};
\end{tikzpicture}\right)^{\circ n}$, i.e.\ a string with $n$ dots.

\begin{lem}
	For each $n \in \N$, we have 
	\begin{equation} \label{eq:multiDotCrossingSlide}
	\begin{tikzpicture}[anchorbase]
	\draw[-] (0, 0) -- (0.8, 0.8);
	\draw[-] (0.8, 0) -- (0, 0.8);
	\adot{(0.6, 0.2)} node[anchor=west, color=black] {$\scr n$};
	\end{tikzpicture}
	\ =\
	\begin{tikzpicture}[anchorbase]
	\draw[-] (0, 0) -- (0.8, 0.8);
	\draw[-] (0.8, 0) -- (0, 0.8);
	\adot{(0.2, 0.6)} node[anchor=east, color=black] {$\scr n$};
	\end{tikzpicture}
	\ + \summ_{i = 1}^{n}\
	\begin{tikzpicture}[anchorbase]
	\draw[-] (0, 0) -- (0, 0.3) arc(180:0:0.3) -- (0.6, 0);
	\draw[-] (0, 1.5) -- (0, 1.2) arc(180:360:0.3) -- (0.6, 1.5);
	\adot{(0.6, 0.3)} node[anchor=west, color=black] {$\scr n - i$};
	\adot{(0, 1.25)} node[anchor=east, color=black] {$\scr i - 1$};
	\teleportRR{0.3, 0.6}{0.3, 0.9};
	\end{tikzpicture}
	\ -\
	\begin{tikzpicture}[anchorbase]
	\draw[-] (0, 0) -- (0, 1);
	\draw[-] (0.6, 0) -- (0.6, 1);
	\teleportUU{0, 0.5}{0.6, 0.5};
	\adot{(0, 0.75)} node[anchor=east, color=black] {$\scr i - 1$};
	\adot{(0.6, 0.25)} node[anchor=west, color=black] {$\scr n - i$};
	\end{tikzpicture}.
	\end{equation}
	(When $n = 0$, the sum over $i$ should be interpreted as 0.)
\end{lem}

\begin{proof}
	We proceed by induction on $n$. For $n = 0$, the identity holds trivially. Assuming that \eqref{eq:multiDotCrossingSlide} holds for some $n \in \N$, we compute:
	\begin{multline*}
	\begin{tikzpicture}[anchorbase]
	\draw[-] (0, 0) -- (0.8, 0.8);
	\draw[-] (0.8, 0) -- (0, 0.8);
	\adot{(0.6, 0.2)} node[anchor=west, color=black] {$\scr n + 1$};
	\end{tikzpicture}
	=
	\begin{tikzpicture}[anchorbase]
	\draw[-] (0, 0) -- (1, 1);
	\draw[-] (1, 0) -- (0, 1);
	\adot{(0.8, 0.2)} node[anchor=west, color=black] {};
	\adot{(0.6, 0.4)} node[anchor=west, color=black] {$\scr n$};
	\end{tikzpicture}
	\overset{\eqref{eq:multiDotCrossingSlide}}{=}
	\begin{tikzpicture}[anchorbase]
	\draw[-] (0, 0) -- (0.8, 0.8);
	\draw[-] (0.8, 0) -- (0, 0.8);
	\adot{(0.2, 0.6)} node[anchor=east, color=black] {$\scr n$};
	\adot{(0.6, 0.2)} node[anchor=east, color=black] {};
	\end{tikzpicture}
	\ + \summ_{i = 1}^{n}\
	\begin{tikzpicture}[anchorbase]
	\draw[-] (0, -0.1) -- (0, 0.3) arc(180:0:0.3) -- (0.6, -0.1);
	\draw[-] (0, 1.5) -- (0, 1.2) arc(180:360:0.3) -- (0.6, 1.5);
	\adot{(0.6, 0.3)} node[anchor=west, color=black] {$\scr n - i$};
	\adot{(0, 1.25)} node[anchor=east, color=black] {$\scr i - 1$};
	\adot{(0.6, 0.1)} node[anchor=east, color=black] {};
	\teleportRR{0.3, 0.6}{0.3, 0.9};
	\end{tikzpicture}
	\ -\
	\begin{tikzpicture}[anchorbase]
	\draw[-] (0, -0.2) -- (0, 1);
	\draw[-] (0.6, -0.2) -- (0.6, 1);
	\teleportUU{0, 0.5}{0.6, 0.5};
	\adot{(0, 0.75)} node[anchor=east, color=black] {$\scr i - 1$};
	\adot{(0.6, 0.25)} node[anchor=west, color=black] {$\scr n - i$};
	\adot{(0.6, 0.05)} node[anchor=east, color=black] {};
	\end{tikzpicture}
	\\
	\overset{\eqref{rel:dotCrossingSlide}}{=}
	\begin{tikzpicture}[anchorbase]
	\draw[-] (0, 0) -- (1, 1);
	\draw[-] (1, 0) -- (0, 1);
	\adot{(0.4, 0.6)} node[anchor=west, color=black] {};
	\adot{(0.2, 0.8)} node[anchor=west, color=black] {$\scr n$};
	\end{tikzpicture}
	\ +\
	\begin{tikzpicture}[anchorbase]
	\draw[-] (0, 0) -- (0, 0.3) arc(180:0:0.3) -- (0.6, 0);
	\draw[-] (0, 1.5) -- (0, 1.2) arc(180:360:0.3) -- (0.6, 1.5);
	\adot{(0, 1.25)} node[anchor=east, color=black] {$\scr n$};
	\teleportRR{0.3, 0.6}{0.3, 0.9};
	\end{tikzpicture}
	\ -\
	\begin{tikzpicture}[anchorbase]
	\draw[-] (0, 0) -- (0, 1);
	\draw[-] (0.6, 0) -- (0.6, 1);
	\teleportUU{0, 0.5}{0.6, 0.5};
	\adot{(0, 0.75)} node[anchor=east, color=black] {$\scr n$};
	\end{tikzpicture}
	\ + \summ_{i = 1}^{n}\
	\begin{tikzpicture}[anchorbase]
	\draw[-] (0, -0.1) -- (0, 0.3) arc(180:0:0.3) -- (0.6, -0.1);
	\draw[-] (0, 1.5) -- (0, 1.2) arc(180:360:0.3) -- (0.6, 1.5);
	\adot{(0.6, 0.3)} node[anchor=west, color=black] {$\scr n - i$};
	\adot{(0, 1.25)} node[anchor=east, color=black] {$\scr i - 1$};
	\adot{(0.6, 0.1)} node[anchor=east, color=black] {};
	\teleportRR{0.3, 0.6}{0.3, 0.9};
	\end{tikzpicture}
	\ -\
	\begin{tikzpicture}[anchorbase]
	\draw[-] (0, -0.2) -- (0, 1);
	\draw[-] (0.6, -0.2) -- (0.6, 1);
	\teleportUU{0, 0.5}{0.6, 0.5};
	\adot{(0, 0.75)} node[anchor=east, color=black] {$\scr i - 1$};
	\adot{(0.6, 0.25)} node[anchor=west, color=black] {$\scr n - i$};
	\adot{(0.6, 0.05)} node[anchor=east, color=black] {};
	\end{tikzpicture}
	\\
	=
	\begin{tikzpicture}[anchorbase]
	\draw[-] (0, 0) -- (0.8, 0.8);
	\draw[-] (0.8, 0) -- (0, 0.8);
	\adot{(0.2, 0.6)} node[anchor=east, color=black] {$\scr n + 1$};
	\end{tikzpicture}
	\ + \summ_{i = 1}^{n + 1}\
	\begin{tikzpicture}[anchorbase]
	\draw[-] (0, 0) -- (0, 0.3) arc(180:0:0.3) -- (0.6, 0);
	\draw[-] (0, 1.5) -- (0, 1.2) arc(180:360:0.3) -- (0.6, 1.5);
	\adot{(0.6, 0.3)} node[anchor=west, color=black] {$\scr n + 1 - i$};
	\adot{(0, 1.25)} node[anchor=east, color=black] {$\scr i - 1$};
	\teleportRR{0.3, 0.6}{0.3, 0.9};
	\end{tikzpicture}
	\ -\
	\begin{tikzpicture}[anchorbase]
	\draw[-] (0, 0) -- (0, 1);
	\draw[-] (0.6, 0) -- (0.6, 1);
	\teleportUU{0, 0.5}{0.6, 0.5};
	\adot{(0, 0.75)} node[anchor=east, color=black] {$\scr i - 1$};
	\adot{(0.6, 0.25)} node[anchor=west, color=black] {$\scr n + 1 - i$};
	\end{tikzpicture},
	\end{multline*}
	as desired.
\end{proof}

The following result is a Frobenius superalgebra analogue of \cite[Lem.~3.4]{ruiSong2018affine}.

\begin{lem} \label{lem:dotBubbleRewriting}
	For each $n \in \N$ and $a \in A$, we have
	\begin{equation} \label{eq:dotBubbleRewriting}
	\begin{tikzpicture}[anchorbase]
	\draw[-] (0, 0.25) arc(360:0:0.375);
	\adot{(-0.075, 0.48)} node[anchor=west, color=black] {$\scr n$};
	\token{0,  0.25}{east}{a};
	\node at (-0.075, -0.06) {\vphantom{$\scr n$}};
	\end{tikzpicture}
	\ =\
	(-1)^{n} \begin{tikzpicture}[anchorbase]
	\draw[-] (0, 0.25) arc(360:0:0.375);
	\adot{(-0.075, 0.48)} node[anchor=west, color=black] {$\scr n$};
	\token{0,  0.25}{east}{a^\star};
	\node at (-0.075, -0.06) {\vphantom{$\scr n$}};
	\end{tikzpicture}
	\ +\
	\summ_{i = 1}^{n} (-1)^{i - 1} \left(\begin{tikzpicture}[anchorbase]
	\draw[-] (0, 0) -- (0, 0.4) arc(180:0:0.2) -- (0.4, 0) arc(360:180:0.2);
	\draw[-] (0, 1.1) arc(180:540:0.2);
	\teleportRR{0.2, 0.6}{0.2, 0.9};
	\adot{(0.4, 0.4)} node[anchor=west, color=black] {$\scr n - i$};
	\token{0.4,  0}{west}{a};
	\adot{(0.4, 1.1)} node[anchor=west, color=black] {$\scr i - 1$};
	\end{tikzpicture}	
	\ -\
	\begin{tikzpicture}[anchorbase]
	\draw[-] (0, 0) -- (0, 0.6) arc(180:0:0.2) -- (0.4, 0) arc(360:180:0.2);
	\teleportUU{0, 0.3}{0.4, 0.3};
	\adot{(0.4, 0.5)} node[anchor=west, color=black] {$\scr n - 1$};
	\token{0.4,  0.1}{west}{a};
	\end{tikzpicture}
	\right).
	\end{equation}
	(When $n = 0$, the sum over $i$ should be interpreted as 0.)
\end{lem}

\begin{proof}
	We compute:
	\begin{multline*}
	\begin{tikzpicture}[anchorbase]
	\draw[-] (0, 0.25) arc(360:0:0.375);
	\adot{(-0.075, 0.48)} node[anchor=west, color=black] {$\scr n$};
	\token{0,  0.25}{east}{a};
	\node at (-0.075, -0.06) {\vphantom{$\scr n$}};
	\end{tikzpicture}
	\overset{\eqref{rel:brauer}}{=}
	\begin{tikzpicture}[anchorbase]
	\draw[-] (0, -0.2) -- (0, 0.2) to[out=90, in=270] (0.4, 0.6) arc(0:180:0.2) to[out=270, in=90] (0.4, 0.2) -- (0.4, -0.2) arc(360:180:0.2);
	\adot{(0.4, 0.2)} node[anchor=west, color=black] {$\scr n$};
	\token{0.4,  -0.2}{west}{a};
	\end{tikzpicture}
	\overset{\eqref{eq:multiDotCrossingSlide}}{=}
	\begin{tikzpicture}[anchorbase]
	\draw[-] (0, -0.2) -- (0, 0.2) to[out=90, in=270] (0.4, 0.6) arc(0:180:0.2) to[out=270, in=90] (0.4, 0.2) -- (0.4, -0.2) arc(360:180:0.2);
	\token{0.4,  -0.2}{west}{a};
	\adot{(0, 0.6)} node[anchor=east, color=black] {$\scr n$};
	\end{tikzpicture}
	\ + \summ_{i = 1}^n\
	\begin{tikzpicture}[anchorbase]
	\draw[-] (0, 0) -- (0, 0.4) arc(180:0:0.2) -- (0.4, 0) arc(360:180:0.2);
	\draw[-] (0, 1.1) arc(180:540:0.2);
	\teleportRR{0.2, 0.6}{0.2, 0.9};
	\adot{(0.4, 0.4)} node[anchor=west, color=black] {$\scr n - i$};
	\token{0.4,  0}{west}{a};
	\adot{(0, 1.1)} node[anchor=east, color=black]{$\scr i - 1$};
	\end{tikzpicture}
	\ -\
	\begin{tikzpicture}[anchorbase]
	\draw[-] (0, 0) -- (0, 0.8) arc(180:0:0.2) -- (0.4, 0) arc(360:180:0.2);
	\teleportUU{0, 0.5}{0.4, 0.5};
	\adot{(0, 0.7)} node[anchor=east, color=black] {$\scr i - 1$};
	\token{0.4,  0}{west}{a};
	\adot{(0.4, 0.3)} node[anchor=west, color=black] {$\scr n - i$};
	\end{tikzpicture}
	\\
	\underset{\eqref{rel:dotTokenCommuting}}{\overset{\eqref{rel:dotCapSlide}}{=}}
	(-1)^{n}
	\begin{tikzpicture}[anchorbase]
	\draw[-] (0, -0.2) -- (0, 0.2) to[out=90, in=270] (0.4, 0.6) arc(0:180:0.2) to[out=270, in=90] (0.4, 0.2) -- (0.4, -0.2) arc(360:180:0.2);
	\token{0.4,  -0.2}{west}{a};
	\adot{(0.4, 0.6)} node[anchor=west, color=black] {$\scr n$};
	\end{tikzpicture}
	\ + \summ_{i = 1}^n (-1)^{i - 1}\left(
	\begin{tikzpicture}[anchorbase]
	\draw[-] (0, 0) -- (0, 0.4) arc(180:0:0.2) -- (0.4, 0) arc(360:180:0.2);
	\draw[-] (0, 1.1) arc(180:540:0.2);
	\teleportRR{0.2, 0.6}{0.2, 0.9};
	\adot{(0.4, 0.4)} node[anchor=west, color=black] {$\scr n - i$};
	\token{0.4,  0}{west}{a};
	\adot{(0.4, 1.1)} node[anchor=west, color=black]{$\scr i - 1$};
	\end{tikzpicture}
	\ -\
	\begin{tikzpicture}[anchorbase]
	\draw[-] (0, 0) -- (0, 0.6) arc(180:0:0.2) -- (0.4, 0) arc(360:180:0.2);
	\teleportUU{0, 0.3}{0.4, 0.3};
	\adot{(0.4, 0.5)} node[anchor=west, color=black] {$\scr n - 1$};
	\token{0.4,  0.1}{west}{a};
	\end{tikzpicture}
	\right)
	\\
	\underset{\eqref{rel:brauerBasics}}{\overset{\eqref{rel:tokens}}{=}}
	(-1)^{n}
	\begin{tikzpicture}[anchorbase, yscale=-1]
	\draw[-] (0, -0.2) -- (0, 0.2) to[out=90, in=270] (0.4, 0.6) arc(0:180:0.2) to[out=270, in=90] (0.4, 0.2) -- (0.4, -0.2) arc(360:180:0.2);
	\adot{(0.4, -0.25)} node[anchor=west, color=black] {$\scr n$};
	\token{0.4,  0.2}{west}{a^\star};
	\end{tikzpicture}
	\ + \summ_{i = 1}^n (-1)^{i - 1}\left(
	\begin{tikzpicture}[anchorbase]
	\draw[-] (0, 0) -- (0, 0.4) arc(180:0:0.2) -- (0.4, 0) arc(360:180:0.2);
	\draw[-] (0, 1.1) arc(180:540:0.2);
	\teleportRR{0.2, 0.6}{0.2, 0.9};
	\adot{(0.4, 0.4)} node[anchor=west, color=black] {$\scr n - i$};
	\token{0.4,  0}{west}{a};
	\adot{(0.4, 1.1)} node[anchor=west, color=black]{$\scr i - 1$};
	\end{tikzpicture}
	\ -\
	\begin{tikzpicture}[anchorbase]
	\draw[-] (0, 0) -- (0, 0.6) arc(180:0:0.2) -- (0.4, 0) arc(360:180:0.2);
	\teleportUU{0, 0.3}{0.4, 0.3};
	\adot{(0.4, 0.5)} node[anchor=west, color=black] {$\scr n - 1$};
	\token{0.4,  0.1}{west}{a};
	\end{tikzpicture}
	\right)\\
	\overset{\eqref{rel:brauerBasics}}{=}
	(-1)^{n} \begin{tikzpicture}[anchorbase]
	\draw[-] (0, 0.25) arc(360:0:0.375);
	\adot{(-0.075, 0.48)} node[anchor=west, color=black] {$\scr n$};
	\token{0,  0.25}{east}{a^\star};
	\node at (-0.075, -0.06) {\vphantom{$\scr n$}};
	\end{tikzpicture}
	\ +\
	\summ_{i = 1}^{n} (-1)^{i - 1} \left(\begin{tikzpicture}[anchorbase]
	\draw[-] (0, 0) -- (0, 0.4) arc(180:0:0.2) -- (0.4, 0) arc(360:180:0.2);
	\draw[-] (0, 1.1) arc(180:540:0.2);
	\teleportRR{0.2, 0.6}{0.2, 0.9};
	\adot{(0.4, 0.4)} node[anchor=west, color=black] {$\scr n - i$};
	\token{0.4,  0}{west}{a};
	\adot{(0.4, 1.1)} node[anchor=west, color=black] {$\scr i - 1$};
	\end{tikzpicture}	
	\ -\
	\begin{tikzpicture}[anchorbase]
	\draw[-] (0, 0) -- (0, 0.6) arc(180:0:0.2) -- (0.4, 0) arc(360:180:0.2);
	\teleportUU{0, 0.3}{0.4, 0.3};
	\adot{(0.4, 0.5)} node[anchor=west, color=black] {$\scr n - 1$};
	\token{0.4,  0.1}{west}{a};
	\end{tikzpicture}
	\right),
	\end{multline*}
	as desired.
\end{proof}

Note that taking $n = 0$ in \eqref{eq:dotBubbleRewriting} yields the identity $\begin{tikzpicture}[anchorbase]
\draw[-] (0, 0.25) arc(360:0:0.375);
\token{0,  0.25}{east}{a};
\end{tikzpicture} = \begin{tikzpicture}[anchorbase]
\draw[-] (0, 0.25) arc(360:0:0.375);
\token{0,  0.25}{east}{a^\star};
\end{tikzpicture}$. This special case previously appeared in \cite[(9.6)]{diagrammaticsRealSupergroups} (in the context of non-affine Frobenius Brauer categories).

Since $-^\star$ is an involution, we have that $A$ decomposes as the direct sum of its $1$-eigenspace and its $(-1)$-eigenspace, which we respectively denote $E_1$ and $E_{-1}$. Taking either $n$ to be odd and $a \in E_1$, or $n$ to be even and $a \in E_{-1}$, \eqref{eq:dotBubbleRewriting} yields
$$
2\ \begin{tikzpicture}[anchorbase]
\draw[-] (0, 0.25) arc(360:0:0.375);
\adot{(-0.075, 0.48)} node[anchor=west, color=black] {$\scr n$};
\token{0,  0.25}{east}{a};
\node at (-0.075, -0.06) {\vphantom{$\scr n$}};
\end{tikzpicture}
\ =\
\summ_{i = 1}^{n} (-1)^{i - 1} \left(\begin{tikzpicture}[anchorbase]
\draw[-] (0, 0) -- (0, 0.4) arc(180:0:0.2) -- (0.4, 0) arc(360:180:0.2);
\draw[-] (0, 1.1) arc(180:540:0.2);
\teleportRR{0.2, 0.6}{0.2, 0.9};
\adot{(0.4, 0.4)} node[anchor=west, color=black] {$\scr n - i$};
\token{0.4,  0}{west}{a};
\adot{(0.4, 1.1)} node[anchor=west, color=black] {$\scr i - 1$};
\end{tikzpicture}	
\ -\
\begin{tikzpicture}[anchorbase]
\draw[-] (0, 0) -- (0, 0.6) arc(180:0:0.2) -- (0.4, 0) arc(360:180:0.2);
\teleportUU{0, 0.3}{0.4, 0.3};
\adot{(0.4, 0.5)} node[anchor=west, color=black] {$\scr n - 1$};
\token{0.4,  0.1}{west}{a};
\end{tikzpicture}
\right).$$
Note that there are at most $n - 1$ dots on each of the bubbles appearing on the right hand side of this identity. Applying this identity repeatedly with different choices of $n$, one can express an arbitrary bubble $\begin{tikzpicture}[anchorbase]
\draw[-] (0, 0.25) arc(360:0:0.375);
\adot{(-0.075, 0.48)} node[anchor=west, color=black] {$\scr n$};
\token{0,  0.25}{east}{a};
\node at (-0.075, -0.06) {\vphantom{$n$}};
\end{tikzpicture}$ as a linear combination of bubbles with an even number of dots and a token labelled by an element of $E_1$, and bubbles with an odd number of dots and a token labelled by an element of $E_{-1}$. This is why we restrict our attention to bubbles of those types in the basis conjecture that follows. Note that when $-^\star = \id$, we have $E_1 = A$ and $E_{-1} = \{0\}$, and hence one can rewrite any bubble with an odd number of dots as a linear combination of bubbles with an even number of dots. This is reflected in the basis theorem for the affine Brauer category $\AB = \AB(\kk, \id)$, namely \cite[Thm.~B]{ruiSong2018affine}, in which all bubbles appearing in basis diagrams have an even number of dots.

For $r, s \in \N$, an $(r, s)$\emph{-Brauer diagram} (in normal form) is a string diagram representing a morphism in $\Hom_{\AB(A, -^\star)}(\go^{\otimes r}, \go^{\otimes s})$ such that:
\begin{itemize}
	\item There are no closed strings, i.e.\ no strings without endpoints;
	\item There are no tokens or dots on any string;
	\item No string has more than one critical point;
	\item No string intersects itself;
	\item No two strings cross each other more than once.
\end{itemize}
Each $(r, s)$-Brauer diagram induces a perfect matching of $\{1, \dotsc, r + s\}$ by pairing the endpoints of each string, numbered from left-to-right and bottom-to-top. For instance, the $(4, 6)$-Brauer diagram 
$$
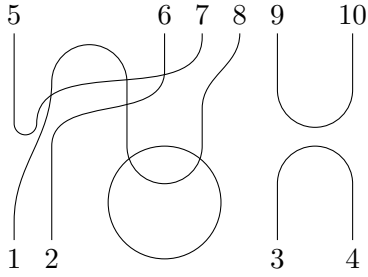
\begin{tikzpicture}[anchorbase]
\draw[-] (0, 0) -- (0.5, 1);
\draw[-] (0, 1) arc(180:360:0.5);
\draw[-] (1.5, 0) -- (1.5, 1);
\draw[-] (2, 0) arc(180:0:0.25);
\draw[-] (2, 1) arc(180:360:0.25);
\node[] at (0, -0.25) {$1$};
\node[] at (1.5, -0.25) {$2$};
\node[] at (2, -0.25) {$3$};
\node[] at (2.5, -0.25) {$4$};
\node[] at (0, 1.25) {$5$};
\node[] at (0.5, 1.25) {$6$};
\node[] at (1, 1.25) {$7$};
\node[] at (1.5, 1.25) {$8$};
\node[] at (2, 1.25) {$9$};
\node[] at (2.5, 1.25) {$10$};
\end{tikzpicture}
$$
yields the matching $\{\{1, 6\}, \{2, 8\}, \{3, 4\}, \{5, 7\}, \{9, 10\} \}$. We define an equivalence relation on the set of $(r, s)$-Brauer diagrams by asserting that two diagrams are equivalent precisely when they induce the same matching of $\{1, \dotsc, r + s\}$. For $r, s \in \N$, let $\D(r, s)$ be a complete set of representatives for this equivalence relation on $(r, s)$-Brauer diagrams. Equivalent diagrams are equal as morphisms in $\AB(A, -^\star)$; see \cite[Prop.~2.5]{ruiSong2018affine}.

Fix a homogeneous basis $B_1$ for the eigenspace $E_1$ of $-^\star$, and a homogeneous basis $B_{-1}$ for the eigenspace $E_{-1}$.

\begin{defin}
	An \emph{admissible bubble} is a string diagram of the following form, representing a morphism in $\Hom_{\AB(A, -^\star)}(\one, \one)$:
	\begin{equation*}
	\begin{tikzpicture}[anchorbase]
	\draw[-] (0, 0.25) arc(360:0:0.375);
	\adot{(-0.075, 0.48)} node[anchor=west, color=black] {$\scr n$};
	\token{0,  0.25}{east}{a};
	\node at (-0.075, -0.06) {\vphantom{$\scr n$}};
	\end{tikzpicture},
	\end{equation*}
	where either $n$ is even and $a \in B_1$, or $n$ is odd and $a \in B_{-1}$. An \emph{admissible bubble chain} is a (potentially empty) set of admissible bubbles arranged in a horizontal line, such that all tokens on bubbles are at the same height. The bubbles in a chain may be ordered arbitrarily from left to right, i.e.\ we consider chains up to permutation of the constituent bubbles.
\end{defin}

For all $a, b \in A$ and $m, n \in \N$ we have:
$$
\begin{tikzpicture}[anchorbase]
\draw[-] (0, 0.25) arc(360:0:0.375);
\adot{(-0.075, 0.48)} node[anchor=west, color=black] {$\scr m$};
\token{0,  0.25}{east}{a};
\node at (-0.075, -0.06) {\vphantom{$\scr n$}};
\draw[-] (1.25, 0.25) arc(360:0:0.375);
\adot{(1.175, 0.48)} node[anchor=west, color=black] {$\scr n$};
\token{1.25,  0.25}{east}{b};
\end{tikzpicture}
\overset{\eqref{eq:superInterchangeLaw}}{=}
\begin{tikzpicture}[anchorbase]
\draw[-] (0, 1.25) arc(360:0:0.375);
\adot{(-0.075, 1.48)} node[anchor=west, color=black] {$\scr m$};
\token{0,  1.25}{east}{a};
\draw[-] (1.25, 0.25) arc(360:0:0.375);
\adot{(1.175, 0.48)} node[anchor=west, color=black] {$\scr n$};
\token{1.25,  0.25}{east}{b};
\end{tikzpicture}
=
\begin{tikzpicture}[anchorbase]
\draw[-] (1.25, 1.25) arc(360:0:0.375);
\adot{(1.175, 1.48)} node[anchor=west, color=black] {$\scr m$};
\token{1.25,  1.25}{east}{a};
\draw[-] (0, 0.25) arc(360:0:0.375);
\adot{(-0.075, 0.48)} node[anchor=west, color=black] {$\scr n$};
\token{0,  0.25}{east}{b};
\end{tikzpicture}
\overset{\eqref{eq:superInterchangeLaw}}{=}
(-1)^{\bar{a}\bar{b}}
\begin{tikzpicture}[anchorbase]
\draw[-] (0, 0.25) arc(360:0:0.375);
\adot{(-0.075, 0.48)} node[anchor=west, color=black] {$\scr n$};
\token{0,  0.25}{east}{b};
\node at (-0.075, -0.06) {\vphantom{$\scr n$}};
\draw[-] (1.25, 0.25) arc(360:0:0.375);
\adot{(1.175, 0.48)} node[anchor=west, color=black] {$\scr m$};
\token{1.25,  0.25}{east}{a};
\end{tikzpicture},
$$
and hence permuting a bubble chain changes the represented morphism by at most a sign. One could eliminate these signs by imposing conventions on the ordering of bubbles in chains, but this is not needed for our purposes.

For $r, s \in \N$, let $\D^{\bullet \circ}(r, s)$ be the set of all morphisms in $\AB(A, -^\star)$ obtained by taking an element of $\D(r, s)$ and performing the following steps:
\begin{itemize}
	\item Add a nonnegative number of dots, and one token labelled by an element of $B_1$ or $B_{-1}$, to one end of each string;
	\item Add an admissible bubble chain to the right of the Brauer diagram.
\end{itemize}
These added bubbles, dots, and tokens are subject to the following conventions:
\begin{itemize}
	\item If both ends of a string are on the top of the diagram, its token and dots appear near the left endpoint of the string, above all crossings, cups, and caps;
	\item If both ends of a string are on the bottom of the diagram, its token and dots appear near the right endpoint of the string, below all crossings, cups, and caps;
	\item If a string has ends at both the top and bottom of the diagram, its token and dots appear near the bottom endpoint the string, below all crossings, cups, and caps;
	\item All tokens near top endpoints are at the same height, and all tokens near bottom endpoints are at the same height;
	\item All dots appear just above the token on its string;
	\item All bubbles are vertically positioned strictly between any cups and caps appearing in the underlying $(r, s)$-Brauer diagram;
\end{itemize}
For instance, $\D^{\bullet \circ}(4, 6)$ contains (among many others) the elements
$$
\begin{tikzpicture}[anchorbase]
\draw[-] (0, 0) -- (0.5, 2);
\draw[-] (0, 2) arc(180:360:0.5);
\token{0.12, 1.67}{east}{a_6};
\adot{(0.01, 1.875)} node[anchor=east, color=black] {$\scr n_6$};
\draw[-] (1.5, 0) -- (1.5, 2);
\draw[-] (2, 0) -- (2, 0.25) arc(180:0:0.25) -- (2.5, 0);
\draw[-] (2, 2) -- (2, 1.75) arc(180:360:0.25) -- (2.5, 2);
\token{0.03125, 0.125}{east}{a_3};
\adot{(0.09375, 0.375)} node[anchor=east, color=black] {$\scr n_3$};
\token{1.5, 0.125}{east}{a_4};
\adot{(1.5, 0.375)} node[anchor=east, color=black] {$\scr n_4$};
\token{2.5, 0.125}{west}{a_5};
\adot{(2.48, 0.375)} node[anchor=west, color=black] {$\scr n_5$};
\token{2.02, 1.67}{east}{a_7};
\adot{(2, 1.875)} node[anchor=east, color=black] {$\scr n_7$};
\draw[-] (3.75, 1) arc(360:0:0.375);
\adot{(3.675, 1.25)} node[anchor=west, color=black] {$\scr 2 n_1$};
\token{3.75, 1}{east}{a_1};
\draw[-] (5.25, 1) arc(360:0:0.375);
\adot{(5.175, 1.25)} node[anchor=west, color=black] {$\scr 2n_2 + 1$};
\token{5.25, 1}{east}{a_2};
\end{tikzpicture},
$$
where $n_1, n_2, n_3, n_4, n_5, n_6, n_7 \in \N$, $a_1 \in B_1$, $a_2 \in B_{-1}$, and $a_3, a_4, a_5, a_6, a_7 \in B_1 \cup B_{-1}$.

\begin{conj} \label{con:basisConjecture}
	For all $r, s \in \N$, $\D^{\bullet\circ}(r, s)$ is a $\kk$-basis for $\Hom_{\AB(A, -^\star)}(\go^{\otimes r}, \go^{\otimes s})$.
\end{conj}

In light of the reductions made possible by Lemma \ref{lem:dotBubbleRewriting}, Conjecture \ref{con:basisConjecture} is the natural statement to expect from the diagrammatic perspective. When $A = \kk$, Conjecture \ref{con:basisConjecture} essentially recovers the basis result for the affine Brauer category $\AB$ proved by Rui and Song in \cite[Thm.~B]{ruiSong2018affine}, though their conventions for the positions of dots and locations of bubbles are slightly different than those in the present paper.

\ 

We conclude by outlining a potential method to prove Conjecture \ref{con:basisConjecture}. First, arguments analogous to those for previously-studied categories yield that $\D^{\bullet\circ}(r, s)$ spans $\Hom_{\AB(A, -^\star)}(\go^{\otimes r}, \go^{\otimes s})$. (The arguments for $\AB$ and $\B(A, -^\star)$ can be found in the proofs of  \cite[Prop~3.7(a)]{ruiSong2018affine} and \cite[Thm.~9.6]{diagrammaticsRealSupergroups}, respectively.) Briefly, the relations \eqref{rel:dotCrossingSlide}, \eqref{rel:dotCapSlide}, \eqref{rel:dotTokenCommuting}, and \eqref{rel:mirroredDotSlide} can be used to move all affine dots to the ends of strings. One can then use the basis theorem for the non-affine category $\B(A, -^\star)$, \cite[Thm.~9.6]{diagrammaticsRealSupergroups}, to manipulate the dot-free parts of diagrams. Finally, \eqref{eq:dotBubbleRewriting} can be used in the manner described above to replace all bubbles with admissible ones, yielding a linear combination of diagrams from $\D^{\bullet\circ}(r, s)$. Showing linear independence is more complicated. The key idea is to embed $\AB(A, -^\star)$ into the additive envelope of a localized version of the oriented Frobenius Brauer category $\AOB(A)$ mentioned in the introduction, and then use the basis result for $\AOB(A)$ (namely the central charge $k = 0$ case of \cite[Thm.~7.2]{Brundan_2021}) to show that the image of $\D^{\bullet\circ}(r, s)$ is linearly independent. A similar method was used to prove \cite[Thm.~9.6]{diagrammaticsRealSupergroups}, which gives homomorphism space bases for the non-affine categories $\B(A, -^\star)$, and \cite[Thm.~5.1]{nilBrauer}, which does the same for the nil-Brauer category $\mathcal{NB}_t$. In turn, these proofs employ arguments similar to those used to prove \cite[Thm.~5.12]{Brundan_2021} and \cite[Thm.~5.4]{degenerateHeisenberg}, respectively; these latter two theorems establish the existence of functors from a Heisenberg category into the additive envelope of a localized symmetric product of two Heisenberg categories.

The author of the current paper has begun investigating the $A = \kk$ case of this proof technique. As noted above, the basis theorem is already known in this case, but the proof of \cite[Thm.~B]{ruiSong2018affine} does not seem like it can be easily generalized to other choices of $A$. The following definitions are for the case $A = \kk$. We assume the reader is already familiar with the definition of the affine oriented Brauer category $\AOB$, which can be found in \cite[Def.~1.2]{ruiSong2018affine}, or presented slightly differently in \cite[Def.~4.3]{mcsween2021affine} (with $A = \kk$).

\begin{defin} \label{def:localizedCategory}
	We define the following morphisms in $\AOB$:
	\begin{equation} \label{def:dottedDumbbells}
	\begin{tikzpicture}[anchorbase]
	\draw[->] (0, 0) -- (0, 0.6);
	\draw[->] (0.4, 0) -- (0.4, 0.6);
	\dottedlink{0, 0.3}{0.4, 0.3};
	\end{tikzpicture}
	\ =\
	\begin{tikzpicture}[anchorbase]
	\draw[->] (0, 0) -- (0, 0.6);
	\draw[->] (0.4, 0) -- (0.4, 0.6);
	\adot{(0, 0.3)};
	\end{tikzpicture}
	+
	\begin{tikzpicture}[anchorbase]
	\draw[->] (0, 0) -- (0, 0.6);
	\draw[->] (0.4, 0) -- (0.4, 0.6);
	\adot{(0.4, 0.3)};
	\end{tikzpicture}, \quad 
	\begin{tikzpicture}[anchorbase]
	\draw[<-] (0, 0) -- (0, 0.6);
	\draw[->] (0.4, 0) -- (0.4, 0.6);
	\dottedlink{0, 0.3}{0.4, 0.3};
	\end{tikzpicture}
	\ =\
	\begin{tikzpicture}[anchorbase]
	\draw[<-] (0, 0) -- (0, 0.6);
	\draw[->] (0.4, 0) -- (0.4, 0.6);
	\adot{(0, 0.3)};
	\end{tikzpicture}
	+
	\begin{tikzpicture}[anchorbase]
	\draw[<-] (0, 0) -- (0, 0.6);
	\draw[->] (0.4, 0) -- (0.4, 0.6);
	\adot{(0.4, 0.3)};
	\end{tikzpicture}, \quad 
	\begin{tikzpicture}[anchorbase]
	\draw[->] (0, 0) -- (0, 0.6);
	\draw[<-] (0.4, 0) -- (0.4, 0.6);
	\dottedlink{0, 0.3}{0.4, 0.3};
	\end{tikzpicture} 
	\ =\
	\begin{tikzpicture}[anchorbase]
	\draw[->] (0, 0) -- (0, 0.6);
	\draw[<-] (0.4, 0) -- (0.4, 0.6);
	\adot{(0, 0.3)};
	\end{tikzpicture} 
	+
	\begin{tikzpicture}[anchorbase]
	\draw[->] (0, 0) -- (0, 0.6);
	\draw[<-] (0.4, 0) -- (0.4, 0.6);
	\adot{(0.4, 0.3)};
	\end{tikzpicture}, \quad
	\begin{tikzpicture}[anchorbase]
	\draw[<-] (0, 0) -- (0, 0.6);
	\draw[<-] (0.4, 0) -- (0.4, 0.6);
	\dottedlink{0, 0.3}{0.4, 0.3};
	\end{tikzpicture}
	\ =\
	\begin{tikzpicture}[anchorbase]
	\draw[<-] (0, 0) -- (0, 0.6);
	\draw[<-] (0.4, 0) -- (0.4, 0.6);
	\adot{(0, 0.3)};
	\end{tikzpicture}
	+
	\begin{tikzpicture}[anchorbase]
	\draw[<-] (0, 0) -- (0, 0.6);
	\draw[<-] (0.4, 0) -- (0.4, 0.6);
	\adot{(0.4, 0.3)};
	\end{tikzpicture}.
	\end{equation}
	
	Let $\L$ denote the category obtained from $\AOB$ by localizing at $\begin{tikzpicture}[anchorbase]
	\draw[->] (0, 0) -- (0, 0.6);
	\draw[->] (0.4, 0) -- (0.4, 0.6);
	\dottedlink{0, 0.3}{0.4, 0.3};
	\end{tikzpicture}$, $\begin{tikzpicture}[anchorbase]
	\draw[<-] (0, 0) -- (0, 0.6);
	\draw[->] (0.4, 0) -- (0.4, 0.6);
	\dottedlink{0, 0.3}{0.4, 0.3};
	\end{tikzpicture}$, $\begin{tikzpicture}[anchorbase]
	\draw[->] (0, 0) -- (0, 0.6);
	\draw[<-] (0.4, 0) -- (0.4, 0.6);
	\dottedlink{0, 0.3}{0.4, 0.3};
	\end{tikzpicture}$, and $\begin{tikzpicture}[anchorbase]
	\draw[<-] (0, 0) -- (0, 0.6);
	\draw[<-] (0.4, 0) -- (0.4, 0.6);
	\dottedlink{0, 0.3}{0.4, 0.3};
	\end{tikzpicture}$. That is, we adjoin two-sided inverses for these morphisms, which we denote $\begin{tikzpicture}[anchorbase]
	\draw[->] (0, 0) -- (0, 0.6);
	\draw[->] (0.4, 0) -- (0.4, 0.6);
	\link{0, 0.3}{0.4, 0.3};
	\end{tikzpicture}$, $\begin{tikzpicture}[anchorbase]
	\draw[<-] (0, 0) -- (0, 0.6);
	\draw[->] (0.4, 0) -- (0.4, 0.6);
	\link{0, 0.3}{0.4, 0.3};
	\end{tikzpicture}$, $\begin{tikzpicture}[anchorbase]
	\draw[->] (0, 0) -- (0, 0.6);
	\draw[<-] (0.4, 0) -- (0.4, 0.6);
	\link{0, 0.3}{0.4, 0.3};
	\end{tikzpicture}$, and $\begin{tikzpicture}[anchorbase]
	\draw[<-] (0, 0) -- (0, 0.6);
	\draw[<-] (0.4, 0) -- (0.4, 0.6);
	\link{0, 0.3}{0.4, 0.3};
	\end{tikzpicture}$, respectively. These morphisms are called \emph{affine teleporters}. This definition immediately implies e.g.\ $ \begin{tikzpicture}[anchorbase]
	\draw[->] (0, 0) -- (0, 0.6);
	\draw[->] (0.4, 0) -- (0.4, 0.6);
	\link{0, 0.4}{0.4, 0.4};
	\adot{(0, 0.2)};
	\end{tikzpicture}
	\ =\
	\begin{tikzpicture}[anchorbase]
	\draw[->] (0, 0) -- (0, 0.6);
	\draw[->] (0.4, 0) -- (0.4, 0.6);
	\end{tikzpicture}
	-
	\begin{tikzpicture}[anchorbase]
	\draw[->] (0, 0) -- (0, 0.6);
	\draw[->] (0.4, 0) -- (0.4, 0.6);
	\link{0, 0.4}{0.4, 0.4};
	\adot{(0.4, 0.2)};
	\end{tikzpicture}\ .$
	
	We also define the following morphisms in $\L$:
	$$\begin{tikzpicture}[anchorbase]
	\draw[->] (0, 0) -- (0, 0.6);
	\bdot{(0, 0.3)};
	\end{tikzpicture} = 2\ \begin{tikzpicture}[anchorbase]
	\draw[->] (0,0) -- (0,0.6) arc(180:0:0.2) -- (0.4,0.4) arc(180:360:0.2) -- (0.8,1);
	\link{0, 0.5}{0.4, 0.5};
	\end{tikzpicture}\ , \quad \quad \begin{tikzpicture}[anchorbase]
	\draw [<-] (0, 0) -- (0, 0.6);
	\bdot{(0, 0.3)};
	\end{tikzpicture} = \begin{tikzpicture}[anchorbase]
	\draw[<-] (0,0) -- (0,0.6) arc(180:0:0.2) -- (0.4,0.4) arc(180:360:0.2) -- (0.8,1);
	\bdot{(0.4, 0.5)};
	\end{tikzpicture}\ .$$
	These are the two-sided inverses of $\begin{tikzpicture}[anchorbase]
	\draw[->] (0, 0) -- (0, 0.6);
	\adot{(0, 0.3)};
	\end{tikzpicture}$ and $\begin{tikzpicture}[anchorbase]
	\draw [<-] (0, 0) -- (0, 0.6);
	\adot{(0, 0.3)};
	\end{tikzpicture}$, respectively.
\end{defin}

\begin{defin} \label{defi:internalBubbles}
	We define the following \emph{internal bubble} morphisms in $\L$:
	\begin{equation} \label{def:internalBubbles}
	\begin{tikzpicture}[anchorbase]
	\draw[->] (0, 0) -- (0, 0.8);
	\purpledot{(0, 0.4)};
	\end{tikzpicture}
	\ =\
	\begin{tikzpicture}[anchorbase]
	\draw[->] (0, 0) -- (0, 0.6);
	\end{tikzpicture}
	+
	\begin{tikzpicture}[anchorbase]
	\draw[->] (0, 0) -- (0, 0.6);
	\draw[->] (0.6, 0.3) arc(360:0:0.2);
	\link{0, 0.3}{0.2, 0.3};
	\end{tikzpicture}
	- \frac{1}{2}\ 
	\begin{tikzpicture}[anchorbase] \draw[->] (0, 0) -- (0, 0.6); \bdot{(0, 0.3)}; \end{tikzpicture}
	- \frac{1}{2}\ 
	\begin{tikzpicture}[anchorbase]
	\draw[->] (0, 0) -- (0, 0.8);
	\draw[->] (0.6, 0.3) arc(360:0:0.2);
	\link{0, 0.3}{0.2, 0.3};
	\bdot{(0, 0.5)};
	\end{tikzpicture}, \quad
	\begin{tikzpicture}[anchorbase]
	\draw[->] (0, 0) -- (0, 0.8);
	\orangedot{(0, 0.4)};
	\end{tikzpicture}
	\ =\
	\begin{tikzpicture}[anchorbase]
	\draw[->] (0, 0) -- (0, 0.6);
	\end{tikzpicture}
	-
	\begin{tikzpicture}[anchorbase]
	\draw[->] (0, 0) -- (0, 0.6);
	\draw[->] (-0.6, 0.3) arc(180:540:0.2);
	\link{-0.2, 0.3}{0, 0.3};
	\end{tikzpicture}
	+ \frac{1}{2}\ 
	\begin{tikzpicture}[anchorbase]
	\draw[->] (0, 0) -- (0, 0.6);
	\bdot{(0, 0.3)};
	\end{tikzpicture}
	-\frac{1}{2}\ 
	\begin{tikzpicture}[anchorbase]
	\draw[->] (0, 0) -- (0, 0.8);
	\draw[->] (-0.6, 0.3) arc(180:540:0.2);
	\link{-0.2, 0.3}{0, 0.3};
	\bdot{(0, 0.5)};
	\end{tikzpicture}, \quad
	\begin{tikzpicture}[anchorbase]
	\draw[<-] (0, 0) -- (0, 0.8);
	\purpledot{(0, 0.4)};
	\end{tikzpicture}
	\ =\
	\begin{tikzpicture}[anchorbase]
	\draw[<-] (0,0) -- (0,0.6) arc(180:0:0.2) -- (0.4,0.4) arc(180:360:0.2) -- (0.8,1);
	\purpledot{(0.4, 0.5)};
	\end{tikzpicture}, \quad
	\begin{tikzpicture}[anchorbase]
	\draw[<-] (0, 0) -- (0, 0.8);
	\orangedot{(0, 0.4)};
	\end{tikzpicture}
	\ =\
	\begin{tikzpicture}[anchorbase]
	\draw[<-] (0,0) -- (0,0.6) arc(180:0:0.2) -- (0.4,0.4) arc(180:360:0.2) -- (0.8,1);
	\orangedot{(0.4, 0.5)};
	\end{tikzpicture}.
	\end{equation}
\end{defin}

Using the relations from Definition \ref{def:localizedCategory}, one can show that the two orientations of internal bubbles are mutual inverses. Our conjectured functor $G \colon \AB \to \text{Add}(\L)$ is defined as follows: the generating object $\go$ of $\AB$ gets sent to $\goup \oplus \godown$, and the generating morphisms are mapped via
	$$
	\begin{tikzpicture}[anchorbase]
	\draw[-] (0, 0) -- (0, 0.7);
	\adot{(0, 0.35)};
	\end{tikzpicture}
	\mapsto
	\begin{tikzpicture}[anchorbase]
	\draw[->] (0, 0) -- (0, 0.7);
	\adot{(0, 0.35)};
	\end{tikzpicture}
	-
	\begin{tikzpicture}[anchorbase]
	\draw[<-] (0, 0) -- (0, 0.7);
	\adot{(0, 0.35)};
	\end{tikzpicture}, \quad \quad 
	\begin{tikzpicture}[anchorbase]
	\draw[-] (0, 0) -- (0, 0.4) arc(180:0:0.3) -- (0.6, 0);
	\end{tikzpicture}
	\mapsto
	\begin{tikzpicture}[anchorbase]
	\draw[->] (0, 0) -- (0, 0.4) arc(0:180:0.3) -- (-0.6, 0); 
	\end{tikzpicture}
	+
	\begin{tikzpicture}[anchorbase]
	\draw[->] (0, 0) -- (0, 0.4) arc(180:0:0.3) -- (0.6, 0);
	\orangedot{(0, 0.3)};
	\end{tikzpicture}, \quad \quad
	\begin{tikzpicture}[anchorbase]
	\draw[-] (0, 0) -- (0, -0.4) arc(180:360:0.3) -- (0.6, 0);
	\end{tikzpicture} 
	\mapsto
	\begin{tikzpicture}[anchorbase]
	\draw[->] (0, 0) -- (0, -0.4) arc(360:180:0.3) -- (-0.6, 0);
	\end{tikzpicture}
	+
	\begin{tikzpicture}[anchorbase]
	\draw[->] (0, 0) -- (0, -0.4) arc(180:360:0.3) -- (0.6, 0);
	\purpledot{(0.6, -0.3)};
	\end{tikzpicture},$$
	$$\begin{tikzpicture}[anchorbase] 
	\draw[-] (0, 0) -- (1, 1);
	\draw[-] (1, 0) -- (0, 1);
	\end{tikzpicture}
	\mapsto 
	\begin{tikzpicture}[anchorbase] 
	\draw[->] (0, 0) -- (1, 1);
	\draw[->] (1, 0) -- (0, 1);
	\end{tikzpicture} 
	+
	\begin{tikzpicture}[anchorbase] 
	\draw[<-] (0, 0) -- (1, 1);
	\draw[<-] (1, 0) -- (0, 1);
	\end{tikzpicture}
	+ 
	\begin{tikzpicture}[anchorbase] 
	\draw[<-] (0, 0) -- (1, 1);
	\draw[->] (1, 0) -- (0, 1);
	\end{tikzpicture} 
	+ 
	\begin{tikzpicture}[anchorbase] 
	\draw[->] (0, 0) -- (1, 1);
	\draw[<-] (1, 0) -- (0, 1);
	\purpledot{(0.75, 0.75)};
	\orangedot{(0.25, 0.25)};
	\end{tikzpicture}
	+
	\begin{tikzpicture}[anchorbase]
	\draw[->] (0, 0) -- (0, 1);
	\draw[<-] (0.4, 0) -- (0.4, 1);
	\link{0, 0.5}{0.4, 0.5};
	\end{tikzpicture} 
	- 
	\begin{tikzpicture}[anchorbase]
	\draw[<-] (0, 0) -- (0, 1);
	\draw[->] (0.4, 0) -- (0.4, 1);
	\link{0, 0.5}{0.4, 0.5};
	\end{tikzpicture}
	+
	\begin{tikzpicture}[anchorbase]
	\draw[->] (0, 0.15) -- (0, -0.4) arc(180:360:0.3) -- (0.6, 0.15);
	\purpledot{(0.6, -0.3)};
	\draw[->] (-0.2, -0.85) -- (-0.2, -0.3) arc(0:180:0.3) -- (-0.8, -0.85);
	\link{-0.2, -0.3}{0, -0.3};
	\end{tikzpicture}
	-
	\begin{tikzpicture}[anchorbase]
	\draw[<-] (0, 0.15) -- (0, -0.4) arc(180:360:0.3) -- (0.6, 0.15);
	\draw[<-] (-0.2, -0.85) -- (-0.2, -0.3) arc(0:180:0.3) -- (-0.8, -0.85);
	\orangedot{(-0.8, -0.5)};
	\link{-0.2, -0.3}{0, -0.3};
	\end{tikzpicture}\ .$$
The author has confirmed that $G$ preserves almost all of the defining relations for $\AB$; the only one remaining is the braid relation, $\begin{tikzpicture}[anchorbase]
\draw[-] (0, 0) -- (0, 0.2) to[out=90, in=270] (0.4, 0.6) to[out=90, in=270] (0.8, 1) -- (0.8, 1.6);
\draw[-] (0.4, 0) -- (0.4, 0.2) to[out=90, in=270] (0, 0.6) -- (0, 1) to[out=90, in=270] (0.4, 1.4) -- (0.4, 1.6);
\draw[-] (0.8, 0) -- (0.8, 0.6) to[out=90, in=270] (0.4, 1) to[out=90, in=270] (0, 1.4) -- (0, 1.6);
\end{tikzpicture}
\ =\
\begin{tikzpicture}[anchorbase, xscale=-1]
\draw[-] (0, 0) -- (0, 0.2) to[out=90, in=270] (0.4, 0.6) to[out=90, in=270] (0.8, 1) -- (0.8, 1.6);
\draw[-] (0.4, 0) -- (0.4, 0.2) to[out=90, in=270] (0, 0.6) -- (0, 1) to[out=90, in=270] (0.4, 1.4) -- (0.4, 1.6);
\draw[-] (0.8, 0) -- (0.8, 0.6) to[out=90, in=270] (0.4, 1) to[out=90, in=270] (0, 1.4) -- (0, 1.6);
\end{tikzpicture}\ .$
Briefly, the definition of $G$ was obtained by first assuming that, in analogy to the definition of the functor from \cite[Thm.~4.2]{nilBrauer}, $G$ maps the dot, cap, and cup as noted above, where the internal bubbles are two initially-unknown endomorphisms of $\goup$.
Calculations in a cyclotomic quotient of $\AOB$ associated to the minimal polynomials of the actions of $\begin{tikzpicture}[anchorbase]
\draw[->] (0, 0) -- (0, 0.7);
\adot{(0, 0.35)};
\end{tikzpicture}$ and
$\begin{tikzpicture}[anchorbase]
\draw[<-] (0, 0) -- (0, 0.7);
\adot{(0, 0.35)};
\end{tikzpicture}$ then allow one to deduce candidates for the definitions of the internal bubbles and the action of $G$ on the crossing, as given above.

\bibliographystyle{alphaurl}
\bibliography{UnorientedFrobenius}

\end{document}